\documentclass[a4, 12pt]{amsart}
\usepackage{amsthm,enumerate,graphicx,color}
\usepackage{amssymb}
\usepackage{amstext}
\usepackage{amsmath}
\usepackage{amscd}
\usepackage{latexsym}
\usepackage{amsfonts}
\theoremstyle{plain}
\newtheorem{thm}{Theorem}[section]
\newtheorem{thmdef}{Theorem--Definition}[section]
\newtheorem*{thm*}{Theorem}
\newtheorem*{cor*}{Corollary}

\newtheorem{prop}[thm]{Proposition}
\newtheorem{lem}[thm]{Lemma}
\newtheorem{cor}[thm]{Corollary}

\newtheorem*{claim*}{Claim}

\theoremstyle{definition}
\newtheorem{defn}[thm]{Definition}
\newtheorem{ex}[thm]{Example}
\newtheorem{rem}[thm]{Remark}

\newtheorem{ques}[thm]{Question}
\newtheorem{prob}[thm]{Problem}

\theoremstyle{remark}

\numberwithin{equation}{thm}
\def\Hom{\mathrm{Hom}}

\def\Ext{\mathrm{Ext}}

\def\coeff{\mathrm{coeff}}

\def\rank{\operatorname{rank}}

\def\e{\mathrm{e}}
\def\frm{\mathfrak m}

\def\K{\mathrm{K}}
\def\H{\mathrm{H}}

\newcommand{\calX}{\mathcal{X}}

\newcommand{\fkm}{\mathfrak{m}}

\newcommand{\mapright}[1]{%
\smash{\mathop{%
\hbox to 1cm{\rightarrowfill}}\limits^{#1}}}

\newcommand{\mapleft}[1]{%
\smash{\mathop{%
\hbox to 1cm{\leftarrowfill}}\limits_{#1}}}

\def\Supp{\mathrm{Supp}}

\def\height{\mathrm{ht}}

\def\Spec{\operatorname{Spec}}

\def\Syz{\mathrm{Syz}}

\begin{document}

\setlength{\baselineskip}{15pt}
\title{Ulrich ideals and modules over two-dimensional rational singularities}
\pagestyle{plain}
\author{Shiro Goto, Kazuho Ozeki, Ryo Takahashi, Kei-ichi Watanabe, Ken-ichi Yoshida}
\address{S. Goto: Department of Mathematics, School of Science and Technology, Meiji University, 1-1-1 Higashimita, Tama-ku, Kawasaki 214-8571, Japan}
\email{goto@math.meiji.ac.jp}
\address{K. Ozeki: Department of Mathematical Science, Faculty of Science, Yamaguchi University, 1677-1 Yoshida, Yamaguchi 853-8512, Japan}
\email{ozeki@yamaguchi-u.ac.jp}
\address{R. Takahashi: Graduate School of Mathematics, Nagoya University, Furocho, Chikusaku, Nagoya 464-8602, Japan}
\email{takahashi@math.nagoya-u.ac.jp}
\urladdr{http://www.math.nagoya-u.ac.jp/~takahashi/}
\address{K.-i. Watanabe and K. Yoshida: Department of Mathematics, College of Humanities and Sciences, Nihon University, 3-25-40 Sakurajosui, Setagaya-Ku, Tokyo 156-8550, Japan}
\email{watanabe@math.chs.nihon-u.ac.jp}
\email{yoshida@math.chs.nihon-u.ac.jp}
\thanks{This work was partially supported by JSPS Grant-in-Aid for Scientific Research (C) 20540050/22540047/22540054/23540059, JSPS Grant-in-Aid for Young Scientists (B) 22740008/22740026 and by JSPS Postdoctoral Fellowships for Research Abroad}
\begin{abstract}
The main aim of this paper is to 
classify Ulrich ideals and Ulrich modules over 
two-dimensional Gorenstein rational singularities 
(rational double points) from a geometric point of view. 
To achieve this purpose, we introduce the notion of 
(weakly) special Cohen--Macaulay modules with respect to 
ideals, and study the relationship between those modules and 
Ulrich modules with respect to good ideals. 
\end{abstract}
\maketitle
\tableofcontents

\section{Introduction}\label{intro}

In the paper \cite{GOTWY} we established the theory of Ulrich ideals 
and modules with a generalized form. 
The concept of Ulrich modules, or maximally generated maximal Cohen--Macaulay modules
(MGMCM modules) was introduced by \cite{U,BHU}.
In our language,  MGMCM modules are just Ulrich modules with respect 
to the maximal ideal.  
While there are very few MGMCM modules in general, 
any maximal Cohen--Macaulay module over a hypersurface local ring 
of multiplicity (degree) $2$ is a finite direct sum of free modules and 
Ulrich modules.  So, our Ulrich modules include  much more 
members than MGMCM modules.

\par 
To state the main results, let us begin with the definition of 
Ulrich ideals and modules. 
Let $A$ be a Cohen--Macaulay local ring with  maximal ideal 
$\fkm$ and $d = \dim A \geq 0$, 
and let $I \subset A$ be a nonparameter $\fkm$-primary ideal. 
For simplicity, we assume that $I$ contains a parameter ideal 
$Q = (a_1, a_2, \ldots, a_d)$ 
of $A$ as a reduction, that is, $I^{r+1}=QI^r$ for some integer $r \ge 1$. 

\begin{defn}\label{Uideal}
We say that $I$ is an \textit{Ulrich ideal} of $A$ 
if it satisfies the following conditions$:$
\begin{enumerate}
\item[$(1)$] $I^2=QI$. 
\item[$(2)$] $I/I^2$ is a free $A/I$-module.
\end{enumerate}
Let $\mathcal{X}_A$ denote the set of all Ulrich ideals that are not parameter ideals.  
\end{defn} 

For instance, if $(A,\fkm)$ is a Cohen--Macaulay local ring of maximal embedding dimension (\cite{S1}) if and only if $\fkm$ is an Ulrich ideal.
 
\begin{defn}\label{Umodule}
Let $M$ be a nonzero finitely generated $A$-module. 
Then we say that $M$ is an \textit{Ulrich $A$-module with respect to $I$}, 
if the following conditions are satisfied$:$ 
\begin{itemize}
\item[(1)] $M$ is a maximal Cohen--Macaulay $A$-module. 
\item[(2)] $\e_I^0(M)=\ell_A(M/IM)$.
\item[(3)] $M/IM$ is $A/I$-free.
\end{itemize}
Here $\e_I^0(M)$ denotes the multiplicity of $M$ with respect to $I$ 
and $\ell_A(M/IM)$ denotes the length of the $A$-module $M/IM$.
\end{defn}

In \cite{GOTWY}, we proved that all higher syzygy modules $\Syz_A^i(A/I)$ of an Ulrich ideal $I$ 
are Ulrich modules with respect to $I$. 
Moreover, if $A$ is of finite CM-representation type, then $\mathcal{X}_A$ is a finite set. 
Recall here that a Cohen--Macaulay local ring is said to be 
\textit{of finite CM-representation type} if 
there are only a finite number of isomorphism classes of indecomposable maximal Cohen--Macaulay $A$-modules. 
Thus we consider the following natural question. 

\begin{prob} \label{Prob-nat}
Let $(A,\fkm)$ be a Cohen--Macaulay local ring of finite CM-representation type. 
\begin{enumerate}
 \item Classify all Ulrich ideals $I$ of $A$.  
 \item Classify all Ulrich $A$-modules with respect to 
a given $\fkm$-primary ideal $I$.  
 \item Determine all ideals $I$ so that there exists 
an Ulrich $A$-module with respect to $I$.  
\end{enumerate}
\end{prob}

\par 
In \cite[Section 9]{GOTWY}, we gave an answer to the problem as above 
in the case of a one-dimensional Gorenstein local ring of finite CM-representation type by using techniques from representation theory of 
maximal Cohen--Macaulay modules.
We want to give a complete answer to the question as above in the case 
of a two-dimensional Gorenstein local ring of finite CM-representation type. 
Notice that $2$-dimensional Gorenstein local rings of finite 
CM-representation type (over an algebraically closed field of characteristic $0$)
are $2$-dimensional Gorenstein rational singularities.  

\par 
Let us explain the organization of the paper. 
In Section 3, we introduce the notion of weakly special Cohen--Macaulay modules;
let $A$ be a Gorenstein local domain and $I \subset A$ an $\fkm$-primary ideal. 
An maximal Cohen--Macaulay $A$-module $M$ is called a weakly special Cohen--Macaulay 
$A$-module with respect to $I$ if $\mu_A(M)=2 \cdot \rank_A M$ and $M/IM$ is 
$A/I$-free, 
where $\mu_A(M)$ denotes the cardinality of a minimal set of generators of $M$; 
see Definition \ref{Wspdef}. 
Then we prove that $M$ is an Ulrich $A$-module with respect to $I$ and $I$ 
is a good ideal (see Section 2)
if and only if $M$ is a weakly special Cohen--Macaulay $A$-module with respect to $I$
for a Gorenstein local domain $A$ and a nonparameter $\frm$-primary stable ideal $I$; 
see Theorem \ref{UvsWsp} for details. 
As an application, we give a partial answer to the Problem \ref{Prob-nat}(3). 
This implies that $I$ is an Ulrich ideal if and only if 
there exists an Ulrich $A$-module 
with respect to $I$ for any two-dimensional Gorenstein rational singularity. 

\par 
In Section 4, we modify the notion of special Cohen--Macaulay $A$-modules 
introduced by Wunram \cite{Wu}: 
Let $A$ be a two-dimensional rational singularity, and $M$ a maximal Cohen--Macaulay 
$A$-module without free summands. Then $M$ is a special Cohen--Macaulay $A$-module 
with respect to $I$ if and only if $\Ext_A^1(M,A)=0$ and $M/IM$ is $A/I$-free; 
see Definition \ref{SpCM-def}. 
Special Cohen--Macaulay $A$-modules are weakly special Cohen--Macaulay $A$-modules
(but the converse is not true in general). 
The main result in this section is the following theorem, which gives 
a criterion for $I$ (resp. $Z$) to be a special ideal (resp. a special cycle)
in terms of cycles. 

\par \vspace{2mm} \par \noindent 
{\bf Theorem  \ref{Spcycle}.}
Let $Z=\sum_{j=1}^r a_j E_j \ne Z_0$ be an anti-nef cycle 
on the minimal resolution $X \to \Spec A$, 
and put $I=I_Z$. 
Let $Z_0=\sum_{j=1}^r n_j E_j$ denote the fundamental cycle on $X$.  
Then the following conditions are equivalent for every $i, 1\le i\le r$. 
\begin{enumerate}
 \item[$(1)$]  $M_i$ is a special Cohen--Macaulay $A$-module with respect to $I$. 
 \item[$(2)$]  $a_i= n_i \cdot \ell_A(A/I) $. 
 \item[$(3)$]  There exist positive cycles $0 < Y_s \le \ldots \le Y_1 \le Z_0$ 
and anti-nef cycles $Z_1,\ldots,Z_s$ 
so that $Z_k = Z_{k-1}+Y_{k}$ for each $k = 1,\ldots,s$ and 
\[
Z_{k-1}\cdot Y_{k} =0, \quad p_a(Y_{k})=0\quad  \text{and} 
\quad \coeff_{E_i} Y_{k} =n_i
\quad \text{for every $k =1,2,\ldots,s$}, 
\]
where $\coeff_{E_i} W$ stands for the coefficient of $E_i$ in a cycle $W$. 
\end{enumerate}
When this is the case, $\ell_A(A/I)=s+1$ and every $I_{k} := I_{Z_{k}}$ 
is a special ideal. 
Moreover, for every $k = 1,2,\ldots,s$, we obtain that $\Supp(Y_{k})$ is connected,   
$\Supp(Y_{k}) \subset \cup \{E_j \subset \Supp(Y_{k-1}) \,|\,E_jZ_{k-1}=0\}$, 
and $Y_{k}$ is the fundamental cycle on $\Supp(Y_{k})$. 
 
\par 
In Section 5, we give a complete list of Ulrich ideals and Ulrich modules 
with respect to some ideal $I$ for any two-dimensional Gorenstein rational Cohen--Macaulay
singularity.
Main tools are the Riemann--Roch formula, the McKay correspondence and 
results in Section 4.  
The following theorem is the main result in this paper.

\begin{thm} \label{Intro-Main}
Let $A$ be a two-dimensional Gorenstein rational singularity. 
Then the set $\calX_A$ of all nonparameter Ulrich ideals is given by$:$
\begin{center}
\begin{tabular}{cl}
$(A_{2m})$ &  $\{(x,y,z),(x,y^2,z),\ldots,(x,y^m,z)\}$. 
 \\[2mm]
$(A_{2m+1})$ & $\{(x,y,z),(x,y^2,z),\ldots,(x,y^{m+1},z)\}$.  \\[2mm] 
$(D_{2m})$ & $\{(x,y,z),(x,y^2,z),\ldots,(x,y^{m-1},z),$  \\
& ~\quad $(x+\sqrt{-1}y^{m-1},y^{m},z),(x-\sqrt{-1}y^{m-1},y^m,z), (x^2,y,z)\}$. 
\\[2mm]
\end{tabular}

\begin{tabular}{cl}
$(D_{2m+1})$  & $\{(x,y,z),(x,y^2,z),\ldots,(x,y^{m},z),(x^2,y,z)\}$. \\[2mm]
$(E_6)$ & $\{(x,y,z),(x,y^2,z) \}$. \\[2mm]
$(E_7)$ & $\{(x,y,z),(x,y^2,z),(x,y^3,z)\}$.  \\[2mm]
$(E_8)$ & $\{(x,y,z),(x,y^2,z) \}$. 
\end{tabular}
\end{center}
\end{thm}

\par 
In Section 6, we discuss Ulrich ideals of two-dimensional non-Gorenstein 
rational singularities. 
We show that any Ulrich ideal is an integrally closed  
and represented on the minimal resolutuion of singularities, 
and also is a special ideal in the sense of Section 4. 
For instance, any non-Gorenstein cyclic quotient singularity admits a unique 
Ulrich ideal, that is, the maximal ideal; see also Section 7.  

\section{Preliminaries} \label{Pre}
\subsection{Ulrich ideals and modules}

First we recall the notion of good ideals in a Gorenstein local ring. 

\begin{defn}[\textrm{See \cite{GIW}}] \label{Gorgood}
Suppose that $A$ is a Gorenstein local ring. 
Let $I \subset A$ be a nonparameter $\fkm$-primary ideal.  
If $I^2=QI$ holds for some minimal reduction $Q$ of $I$, 
then $I$ is called a \textit{stable} ideal. 
If $I$ is stable and $Q \colon I=I$, then $I$ is called a \textit{good} ideal.
An $\fkm$-primary stable ideal $I$ is good if and only if 
$e_I^0(A)=2 \cdot \ell_A(A/I)$. 
\end{defn}

\par
An Ulrich ideal in a Gorenstein local ring is always a good ideal.

\begin{prop}[\textrm{See \cite[Lemma 2.3, Corollary 2.6]{GOTWY}}] \label{Uchar}
Let $A$ be a $d$-dimensional Cohen--Macaulay local ring, and let $I \subset A$ be 
a nonparameter $\fkm$-primary ideal. Then$:$ 
\begin{enumerate}
 \item[$(1)$] Suppose that $I$ is stable. 
Then $e_I^0(A) \le (\mu_A(I)-d+1)\cdot \ell_A(A/I)$. 
Equality holds if and only if $I$ is an Ulrich ideal. 
 \item[$(2)$] Suppose that $A$ is Gorenstein. 
Then the following conditions are equivalent$:$
\begin{enumerate}
 \item[$(a)$] $I$ is an Ulrich ideal. 
 \item[$(b)$]  $I$ is a good ideal and $\mu_A(I)=d+1$. 
 \item[$(c)$]  $I$ is a good ideal and $A/I$ is Gorenstein. 
\end{enumerate}
\end{enumerate}
\end{prop}

\par
Let us give two typical examples of Ulrich ideals. 

\begin{ex} \label{ex-mUlrich}
It is well-known that $\mu_A(\fkm) \le e_{\fkm}^0(A) + \dim A-1$ holds true. 
Equality holds if and only if the maximal ideal $\fkm$ is stable; see \cite{S1}. 
Then $A$ is said to have  \textit{maximal embedding dimension}.
By \ref{Uchar} (1), $\fkm$ is an Ulrich ideal if and only if $A$ has 
maximal embedding dimension. 

\par 
Suppose that $A$ is a two-dimensional hypersurface of degree $2$. 
Then the maximal ideal $\fkm$ is an Ulrich ideal. 
Moreover, a power $\fkm^k$ is a good ideal but not an Ulrich ideal 
for all $k \ge 2$. 
\end{ex}

\begin{ex} \label{ex-typical}
Let $A=k[[x_0,x_1,\ldots,x_d]]/(x_0^{n_0}+\cdots +x_d^{n_d})$ be a diagonal 
hypersurface.  
Suppose that $n_0=2m$ is even. 
Then $(x_0^{m},x_1^{k_1},\ldots,x_d^{k_d})$ is an Ulrich ideal 
for every $1 \le k_i \le \lfloor \frac{n_i}{2} \rfloor$ $(i=1,2,\ldots,d)$. 
\end{ex}

\par 
The following theorem gives a relationship between Ulrich ideals and Ulrich 
modules with respect to ideals. 

\begin{thm}[\textrm{cf. \cite[Theorem 4.1]{GOTWY}}] \label{Syz}
Let $A$ be a Cohen--Macaulay local ring of dimension $d$. 
Then the following conditions are equivalent$:$
\begin{enumerate}
 \item[$(1)$] $I$ is a nonparameter Ulrich ideal. 
 \item[$(2)$] $\Syz_A^i(A/I)$ is an Ulrich $A$-module with respect to $I$ 
for all $i \ge d$. 
\end{enumerate}
\end{thm}

Note that there exists a non-Ulrich ideal $I$
 so that $\Syz_A^i(A/I)$ is an Ulrich $A$-module 
with respect to $I$; see e.g. Examples \ref{ex-hyp}, \ref{ex-0dimhyp}.

\par 
On the other hand, we can construct 
new Ulrich modules from a given Ulrich module 
by the following theorem.

\begin{thm}[\textrm{See also \cite[Lemma 4.2, Theorem 5.1]{GOTWY}}] \label{Dual}
Suppose that $A$ is a Cohen--Macaulay local ring of dimension $d$ which 
admits a canonical module $K_A$. 
Assume that $I$ is an Ulrich ideal with $\mu(I)> d$ and 
$M$ is an Ulrich $A$-module with respect to $I$.  
Then 
\begin{enumerate}
 \item[$(1)$] $\Syz_A^1(M)$ is an Ulrich $A$-module with respect to $I$. 
 \item[$(2)$] $M^{\vee}=\Hom_A(M,K_A)$ is an Ulrich $A$-module with respect to $I$. 
\end{enumerate}
\end{thm}

\subsection{Two-dimensional rational singularities}

Throughout this subsection, 
let $A$ be a two-dimensional complete normal  local domain with unique
maximal ideal $\fkm$ containing an algebraically closed field $k$ of 
characteristic $0$, unless otherwise specified. 
(Many results in this paper hold true if $k$ is an algebraically closed field of 
positive characteristic. 
For simplicity, we assume that $k$ has characteristic $0$.)
Moreover, assume that $A$ has a \textit{rational singularity}, 
that is, there exists a resolution of singularities $\varphi: X \to \Spec A$ with 
$\H^1(X,\mathcal{O}_X)=0$; see \cite{Li1,Li2}.
A typical example of rational singularities is a quotient singularity. 
Moreover, (two-dimensional) Gorenstein rational singularities are called 
\textit{rational double points}, which are hypersurfaces of degree $2$.

\par \vspace{2mm}
{\bf Positive cycles, anti-nef cycles}. 
In what follows, let $\varphi \colon X \to \Spec A$ be a resolution of singularities 
with $E=\varphi^{-1}(\fkm)$ the exceptional divisor. 
Let $E=\cup_{i=1}^r E_i$ be the decomposition into irreducible components of $E$. 
In the set $\mathcal{C}=\sum_{i=1}^r \mathbb{Z} E_i$ of cycles supported on $E$, 
we define a partial order $\le$ as follows: for $Z$, $Z' \in \mathcal{C}$, 
$Z \le Z'$ if every coefficient of $E_i$ in $Z'-Z$ is nonnegative.
A cycle $Z=\sum_{i=1}^r a_i E_i$ 
is called \textit{positive}, denoted by $Z > 0$, 
if $0 \le Z$ and $Z \ne 0$. 

\par
On the other hand, a positive cycle $Z=\sum_{i=1}a_iE_i$ is said to be 
\textit{anti-nef} if $ZE_i \le 0$ for every $i=1,\ldots,r$, where 
$ZY$ denotes the intersection number of $Z$ and $Y$. 

\par \vspace{2mm} 
 {\bf Virtual genus}.
Since the intersection matrix $[E_i E_j]_{1 \le i,j \le r}$ 
is negative definite, there exists the unique $\mathbb{Q}$-divisor 
$\K_X$, the \textit{canonical divisor}, so that the following equation 
\[
 p_a(E_i):=\dfrac{E_i^2+\K_X E_i}{2} + 1 =0
\]    
holds for every $i=1,\ldots,r$, 
where $K_X$ is the canonical divisor of $X$. 
If $E_i^2=\K_X E_i=-1$, then $E_i \cong \mathbb{P}^1$ 
is called a \textit{$(-1)$-curve}. 
We say that $X$ is a \textit{minimal resolution} if $X$ contains no 
$(-1)$-curve.   
Such a resolution is
unique up to isomorphism. 
Moreover, for any positive cycle $Y > 0$, we put 
\[
 p_a(Y)=\dfrac{Y^2+K_XY}{2} + 1, 
\]
which is called the \textit{virtual genus} of $Y$. 
One can easily see that 
\[
 p_a(Y+Y')=p_a(Y)+p_a(Y')+YY'-1. 
\]
Furthermore, it is well-known that 
if $A$ is a rational singularity then $p_a(Z) \le 0$ holds true 
for every positive cycle $Z$ (\cite[Proposition 1]{Ar}). 

\par \vspace{2mm}
{\bf Dual graph}. In what follows, assume that 
$\varphi \colon X \to \Spec A$ is the minimal 
resolution of singularities with $\varphi^{-1}(\fkm) = \cup_{i=1}^r E_i$. 
Then the \textit{dual graph} $\Gamma$ of $\varphi$ is a simple graph 
with the vertex set
$\{E_i\}_{i=1}^r$ and the edge 
defined by the following:
\[
 \text{the edge $E_i -E_j$ exists (resp. does not exist) if and only if} 
\;\; E_iE_j =1 \;\text{(resp. $E_iE_j=0$)}.
\]
For instance, we have the following example:
\par \vspace{6mm}
\begin{picture}(400,30)(0,0)
    \thicklines
\put(90,10){$\Gamma=$}
  \put(125,0){{\tiny $E_2$}}
\put(130,12){\circle{8}}
\put(135,12){\line(1,0){20}}
  \put(155,0){{\tiny $E_3$}}
\put(160,12){\circle{8}}
\put(165,12){\line(1,0){20}}
  \put(183,0){{\tiny $E_4$}}
\put(190,12){\circle{8}}
\put(195,12){\line(1,0){20}}
  \put(215,0){{\tiny $E_5$}}
\put(220,12){\circle{8}}
\put(225,12){\line(1,0){20}}
  \put(245,0){{\tiny $E_6$}}
\put(250,12){\circle{8}}
\put(190,16){\line(0,1){15}}
  \put(195,26){{\tiny $E_1$}}
\put(190,34){\circle{8}}
\end{picture}
\[
(E_1E_4 = E_2E_3 = E_3 E_4= E_4E_5 = E_5E_6=1,\qquad E_iE_j=0 \text{ (others)}
\]
\par \vspace{3mm}
Let $Y=\sum_{j=1}^r a_jE_j$ be a positive cycle on $X$. 
Then we put $\Supp(Y)=\cup \{E_i \,|\, a_i > 0\}$, the \textit{support} of $Y$. 
Such a set is called connected if the induced subgraph is connected. 
Note that if $Y$ is positive and $p_a(Y)=0$ then $Y$ is connected.

\par \vspace{2mm}
{\bf Integrally closed ideal}.
Let $I$ be an $\fkm$-primary ideal of $A$. 
Then $I$ is said to be \textit{represented} on $X$ if 
the sheaf $I\mathcal{O}_X$ is invertible, that is, there 
exists an anti-nef cycle $Z$ with support in $E$ 
so that $I\mathcal{O}_X = \mathcal{O}_X(-Z)$ and 
$I=\H^0(X,\mathcal{O}_X(-Z))$.  
Then we denote such an ideal $I$ by $I=I_Z$.  
The product of two integrally closed ideals of $A$ 
is also integrally closed (\cite{Li1}). 
There is a one-to-one correspondence between 
the set of integrally closed $\fkm$-primary ideals of $A$ that are represented 
on $X$ and the set of anti-nef cycles 
$Z=\sum_{i=1}^{r} a_iE_i$ on $X$. 

\par \vspace{2mm}
{\bf Good ideal}.
Now we recall the notion of good ideals of rational singularities.

\begin{defn} \label{Ratgood}
Let $I$ be an $\fkm$-primary ideal of $A$. 
Then $I$ is called \textit{good} if $I$ is represented on 
the minimal resolution of singularities. 
\end{defn}

Notice that this definition is different from that of Definition \ref{Gorgood}. 
But for any $\fkm$-primary ideal $I$ of a two-dimensional Gorenstein 
rational singularity, $I$ is good in the sense of Definition \ref{Gorgood} 
if and only if it is good in the sense of Definition \ref{Ratgood};
see also \cite[Theorem 7.8]{GIW} or \cite{WY}).

\par 
The following fact is well-known. 

\begin{lem} \label{good-known}
Let $A$ be a two-dimensional $($not necessarily Gorenstein$)$ 
rational singularity, and $\varphi \colon X \to \Spec A$ denotes 
the minimal resolution of singularities. 
Then$:$ 
\begin{enumerate}
\item[$(1)$] The minimum element $($say, $Z_0$$)$
 among all non-zero anti-nef cycles on $X$ exists.
This cycle $Z_0$ is called the \textit{fundamental cycle} on $X$
which corresponds to the maximal ideal $\fkm$.  
In particular, $\fkm = \H^0(X,\mathcal{O}_X(-Z_0))$ is a good ideal. 
\item[$(2)$] 
If $I=\H^0(X,\mathcal{O}_X(-Z))$ and 
$J = \H^0(X,\mathcal{O}_X(-Z'))$ are good ideals of $A$, then 
$IJ=\H^0(X,\mathcal{O}_X(-(Z+Z')))$ is also a good ideal. 
\item[$(3)$] If $I=\H^0(X,\mathcal{O}_X(-Z))$, then $\e_I^0(A)=-Z^2$.
\end{enumerate} 
\end{lem}

\par 
The colength $\ell_A(A/I)$ can also be determined 
by the anti-nef cycle $Z$; see the
Riemann--Roch formula (Lemma \ref{RR}).

\section{Weakly special Cohen--Macaulay modules over Gorenstein local domains} 
\label{Wsp}

Throughout this section, let $A$ be a Gorenstein local domain 
and $I \subset A$ a nonparameter $\fkm$-primary ideal, 
unless otherwise specified. 
In this section, we introduce the notion of 
weakly special Cohen--Macaulay modules, 
which are closely related to Ulrich modules.

\begin{defn}[\textbf{Weakly special CM module, ideal}] \label{Wspdef}
Let $A$ be a Cohen--Macaulay local domain, and 
let $M$ be an maximal Cohen--Macaulay $A$-module. 
If $M$ satisfies $\mu_A(M)=2 \cdot \rank_A M$ and $M/IM$ is $A/I$-free, then 
$M$ is called a 
\textit{weakly special Cohen--Macaulay $A$-module with respect to} $I$. 

\par
Suppose that $I \subset A$ is a stable ideal. 
If there exists a weakly special Cohen--Macaulay $A$-module 
with respect to $I$, then 
$I$ is called a \textit{weakly special ideal} of $A$.  
\end{defn}

\par
Now suppose that $A$ is a Gorenstein local ring. 
Let $I \subset A$ be a stable $\fkm$-primary ideal with minimal reduction $Q$.   
Then as $I \subseteq Q:I$, we have $I/Q \subseteq (Q:I)/Q \cong K_{A/I}$. 
Hence 
\[
 e_I^0(A) = \ell_A(A/Q)  \le \ell_A(A/I) + \ell(K_{A/I})= 2 \cdot \ell_A(A/I),
\]
where the last equality follows from the Matlis duality theorem. 
Note that equality holds if and only if $I$ is a good ideal. 
\par 
The following theorem is the main result in this section. 

\begin{thm} \label{UvsWsp}
Suppose that $A$ is a Gorenstein local domain and $I$ is a stable ideal of $A$. 
Let $M$ be a maximal Cohen--Macaulay $A$-module. 
Then the following condition are equivalent$:$
\begin{enumerate}
 \item[$(1)$] $M$ is an Ulrich $A$-module with respect to $I$, 
and $I$ is a good ideal. 
 \item[$(2)$] $M$ is a weakly special Cohen--Macaulay 
$A$-module with respect to $I$. 
\end{enumerate}
\end{thm}

\begin{proof}[Proof of Theorem $\ref{UvsWsp}$]
We may assume that $M/IM$ is $A/I$-free. 
Thus 
\[
\ell_A(M/IM) = \mu_A(M) \cdot \ell_A(A/I).
\]
\par 
$(1) \Longrightarrow (2):$
 By assumption we have
\[
 \ell_A(M/IM) 
= e_I^0(M) 
= e_I^0(A) \cdot \rank_A M 
= 2 \cdot \ell_A(A/I) \cdot \rank_A M,
\]
where the second equality follows from the associativity 
formula of multiplicities (e.g. \cite[Theorem 14.8]{Ma}).
It follows from the above two equalities that $\mu_A(M) = 2 \cdot \rank_A M$. 
Thus $M$ is a weakly special Cohen--Macaulay $A$-module with respect to $I$.  

\par 
$(2) \Longrightarrow (1):$ 
Since $M$ is a maximal Cohen--Macaulay $A$-module, we have 
\[
\ell_A(M/IM)  \le  e_I^0(M). 
\]
On the other hand, by the observation and the equality described as above,
we get 
\[
\ell_A(M/IM) 
= \mu_A(M) \cdot \ell_A(A/I) \\
 =  2 \cdot \rank_A M \cdot \ell_A(A/I) \\
 \ge  e_I^0(A) \cdot \rank_A M = e_I^0(M). 
\]
Therefore $\ell_A(M/IM) = e_I^0(M)$ and $e_I^0(A) = 2 \cdot \ell_A(A/I)$. 
That is, $M$ is an Ulrich $A$-module with respect to $I$ 
and $I$ is a good ideal. 
\end{proof}

\begin{cor}\label{Wspideal}
Suppose that $A$ is a Gorenstein local domain. 
If $I$ is an Ulrich ideal, then it is a weakly special ideal. 
\end{cor}

\begin{proof}
If $I$ is an Ulrich ideal, then it is a good ideal by Proposition \ref{Uchar}
and $M=\Syz_A^{\dim A}(A/I)$ is an Ulrich $A$-module 
with respect to $I$ by Theorem \ref{Syz}.
By Theorem \ref{UvsWsp}, $M$ is a weakly special Cohen--Macaulay $A$-module with respect to $I$. 
Hence $I$ is a weakly special ideal, as required. 
\end{proof}

\begin{prop} \label{UvsWspideal}
Suppose that $A$ is a hypersurface local domain. 
Then $I \subset A$ is an Ulrich ideal if and only if it is a weakly special ideal. 
\end{prop}

\begin{proof}
It suffices to prove the `if' part.  
Now suppose that $I$ is a weakly special ideal. 
Take a weakly special Cohen--Macaulay $A$-module $M$ with respect to $I$.
By Theorem \ref{UvsWsp}, $M$ is an Ulrich $A$-module with respect to $I$. 
Since $A$ is a hypersurface and $M$ is a maximal Cohen--Macaulay $A$-module without 
free summands, we have a minimal free presentation  
$A^{\mu} \to A^{\mu} \to M \to 0$, which induces an exact sequence
\[
(A/Q)^{\mu} \to (A/Q)^{\mu} \xrightarrow{f} M/QM \to 0.
\]
As $M/QM=M/IM$ is $A/I$-free, we have $M/QM\cong(A/I)^{\mu}$.
It is easy to observe that the kernel of $f$ is isomorphic to $(I/Q)^{\mu}$.
Hence there is a surjection $(A/Q)^{\mu} \to(I/Q)^{\mu}$, which shows 
$\mu_A(I/Q)\le1$.
Thus $\mu_A(I)=d+1$, and hence
Proposition \ref{Uchar} implies that $I$ is an Ulrich ideal.
\end{proof}

\par 
The following corollary gives a partial answer to Problem \ref{Prob-nat}. 

\begin{cor} \label{UidealvsMod}
Suppose that $A$ is a hypersurface local domain, and $I \subset A$ is a good ideal. 
If there exists an Ulrich $A$-module with respect to $I$, then 
$I$ is an Ulrich ideal. 
\end{cor}

\begin{proof}
The assertion follows from Theorem \ref{UvsWsp} and Proposition \ref{UvsWspideal}.
\end{proof}

\begin{ques} \label{Q-1}
Let $A$ be a Gorenstein local domain and $I \subset A$ be a stable ideal. 
Suppose that there exists an Ulrich $A$-module $M$ with respect to $I$. 
Is then $I$ an Ulrich ideal (especially, a good ideal)?
\end{ques}

\par 
The next examples shows that we cannot relax the assumption 
that $I$ is stable in Question \ref{Q-1}.

\begin{ex} \label{ex-hyp}
Let $k$ be a field and let $e \ge 3$ be an integer. 
Set $A=k[[t^e,t^{e+1}]]$ and $M=(t^e,t^{e+1})^{e-1}$. 
Then $A$ is a hypersurface local domain and $M$ is an Ulrich $A$-module 
with respect to $\fkm=(t^e,t^{e+1})$, the maximal ideal of $A$.
But $\fkm$ is \textit{not} stable. 
\end{ex}

\par 
The next example shows that we cannot relax the assumption 
that $A$ is a local domain (or $\dim A \ge 1$) 
in Question \ref{Q-1}. 

\begin{ex} \label{ex-0dimhyp}
Let $k$ be a field, and let $a$, $e$ be integers with
$2a > e > a \ge 2$. 
Set $A=k[[t]]/(t^{e})$, and $I=(t^a)$. 
Then $I^2=0$ but $I \ne 0 \colon I=(t^{e-a})$. 
Hence $I$ is stable but not good. 
Then $t^{e-a}A \cong A/I$ is an Ulrich $A$-module with 
respect to $I$. 
\end{ex}

\section{Special Cohen--Macaulay modules over two-dimensional 
rational singularities} \label{Sp}

Throughout this section, let $(A,\fkm)$ be a two-dimensional complete 
normal local domain with an algebraically closed residue field $k$ 
of chracteristic zero. 
Let $\varphi \colon X \to \Spec A$ be 
the \textit{minimal} resolution of singularities 
with $E=\varphi^{-1}(\fkm)$ the exceptional divisor. 
Let $E=\cup_{j=1}^r E_j$ be the decomposition into irreducible components of $E$. 
Let $I \subset A$ be an $\fkm$-primary ideal, and $Q$ a minimal reduction of $I$. 
For every maximal Cohen--Macaulay $A$-module $M$, we put 
$\widetilde{M} = \varphi^{*}(M)/\text{torsion}$. 

\par
First we recall the notion of special Cohen--Macaulay modules. 

\begin{thmdef}[\textbf{Special McKay correspondence due to Wunram}]
Assume that $A$ is a rational singularity, and 
let $\varphi \colon X \to \Spec A$ be as above. 
For every $i$, there exists a
unique indecomposable maximal Cohen--Macaulay $A$-module $M_i$ 
$($up to isomorphism$)$ with $H^1(\widetilde{M_i}^{\vee})=0$ so that 
\[
 c_1(\widetilde{M_i}) E_j = \delta_{ij} 
\]
for every $1 \le i,j \le r$ and $\rank_A M_i = n_i$, where  
$c_1(\widetilde{M})$ denotes the 1st Chern class of $\widetilde{M}$
and $Z_0$ denotes the fundamental cycle on $X$. 
\end{thmdef}

Based upon this theorem, we define a (nontrivial) 
special Cohen--Macaulay $A$-module, which has been defined 
in more general settings. 

\begin{defn}[\textbf{Special CM module}] \label{SpCMdef}
Suppose that $A$ is a two-dimensional rational singularity. 
Let $M$ be a maximal Cohen--Macaulay $A$-module. 
Then $M$ is called a \textit{special Cohen--Macaulay $A$-module} 
if $M$ is isomorphic to a finite direct sum of 
$M_1,\ldots, M_r$. 
\end{defn}

\begin{rem} \label{Special-Wunram}
Let $K_A$ denote the canonical module of $A$. 
A maximal Cohen--Macaulay $A$-module $M$ 
is said to be a \textit{special Cohen--Macaulay $A$-module} 
if $M \otimes_A K_A/\text{torsion}$ is Cohen--Macaulay. 
This condition is equivalent to  $\Ext_A^1(M,A)=0$; see \cite{Wu}. 
In particular, any free $A$-module 
or any maximal Cohen--Macaulay module over a 
Gorenstein local domain $A$
is a special Cohen--Macaulay $A$-module in this sense. 
But in this paper, we use the notion of special 
Cohen--Macaulay modules for two-dimensional rational singularities only. 
\end{rem}

\par
Iyama--Wemyss \cite{IW} proved the following characterization of special 
Cohen--Macaulay modules. 

\begin{prop}[\textrm{cf. \cite[Theorem 3.6]{IW}}]
Suppose that $A$ is a two-dimensional rational singularity. 
Let $M$ be a maximal Cohen--Macaulay $A$-module without free summands. 
Then $M$ is a special Cohen--Macaulay $A$-module if and only if 
$\Syz_A^1(M) \cong M^{*}$, where $M^{*} = \Hom_A(M,A)$. 
\end{prop}

\begin{rem}
Suppose that $A$ is Gorenstein rational singularity, 
that is, $A$ is a rational double point. 
Then any maximal Cohen--Macaulay $A$-module is a finite direct sum of 
free modules and special Cohen--Macaulay $A$-modules. 
\end{rem}

As in the case of Ulrich modules, we define a special CM module with respect to an 
Ulrich ideal $I$.
\begin{defn}[\textbf{Special CM module w.r.t. $I$}]
 \label{SpCM-def}
Suppose that $A$ is a two-dimensional rational singularity. 
Let $M$ be a finitely generated $A$-module. 
Then $M$ is called a \textit{special Cohen--Macaulay $A$-module with respect to $I$} 
if the following conditions are satisfied:
\begin{enumerate}
 \item[$(1)$] $M$ is a special Cohen--Macaulay $A$-module, that is, $\Syz_A^1(M) \cong M^{*}$.
 \item[$(2)$]  $M/IM$ is $A/I$-free. 
\end{enumerate}
\end{defn}

\par 
Any special Cohen--Macaulay $A$-module is 
a weakly special Cohen--Macaulay $A$-module 
in the sense of \ref{Gorgood} but we believe that 
the converse is not true in general.  

\begin{lem} \label{WSpvsSp}
Suppose that $A$ is a two-dimensional rational singularity. 
Let $M$ be a maximal Cohen--Macaulay $A$-module. Then
\begin{enumerate}
 \item[$(1)$]  If $M$ is a special Cohen--Macaulay $A$-module with respect to $I$, 
 then it is a weakly special Cohen--Macaulay $A$-module with respect to $I$. 
 \item[$(2)$]  When $\rank_A M =1$, the converse of $(1)$ holds true. 
\end{enumerate}
\end{lem}

\begin{proof}
(1) Suppose that $M$ is a special Cohen--Macaulay $A$-module. 
Then we have the following exact sequence:
\[
 0 \to M^{*} \to A^n \to M \to 0,
\]
where $n= \mu_A(M)$. 
This yields $\mu_A(M) = \rank_A M + \rank_A M^{*} = 2 \cdot \rank_A M$. 
\par \noindent  
(2) Take an ideal $J \subset A$ that is isomorphic to $M$. 
Then $\height J = 1$ and $A/J$ is Cohen--Macaulay. 
It suffices to show that $\Syz_A^1(J) \cong J^{*}$. 
\par 
As $\mu_A(J)=2 \cdot \rank_A J = 2$, we can write $J=(x,y)$. 
Then 
\[
\Syz_A^1(J) \cong \left\{
\left[\begin{array}{c} \alpha \\ \beta \end{array} \right] \in A^2 \,\bigg|\,
\alpha x + \beta y = 0
 \right\} \cong (x) \colon y \cong J^{*}.
\]
Hence $M \cong J$ is a special Cohen--Macaulay $A$-module with respect to $I$. 
\end{proof}

\begin{rem} \label{Converse}
Let $S=k[s,t]$ be a graded 
polynomial ring with two variables over an algebrically closed 
field of characteristic $0$ with $\deg(s)=\deg(t)=1$.  
Let $D$ be an invariant subring of $S$ by 
\[
 G= \bigg\langle \left(
\begin{array}{cc} \sqrt{-1} & 0 \\ 0 & -\sqrt{-1} \end{array} \right), 
\quad \left(\begin{array}{cc} 0 & -1 \\ 1 & 0 \end{array}\right) \bigg\rangle.
\]
Then $D$ is a two-dimensional rational singularity of type $(D_4)$, and 
it is isomorphic to the graded subring $k[x,y,z]$, where 
$x=s^4+t^4$, $y=s^2t^2$, and $z=st(s^4-t^4)$.  
\par
Let $A$ be the third Veronese subring of $D$, that is, 
$A = k[z,x^3,xy^2,y^3]$ is a rational triple point whose dual graph is 
given by the following:
\par 
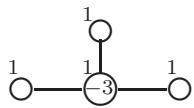
\vspace{5mm}
\begin{picture}(200,35)(-40,0)
    \thicklines
  \put(55,18){{\tiny $1$}}
\put(60,12){\circle{8}}
\put(65,12){\line(1,0){18}}
  \put(83,18){{\tiny $1$}}
\put(90,12){\circle{12}}
\put(84,10){{\tiny $-3$}}
\put(97,12){\line(1,0){18}}
  \put(115,18){{\tiny $1$}}
\put(120,12){\circle{8}}
\put(90,18){\line(0,1){13}}
  \put(83,38){{\tiny $1$}}
\put(90,34){\circle{8}}
\end{picture}
\par \vspace{3mm}
In particular, all indecomposable 
special Cohen--Macaulay $A$-modules have rank $1$. 
\par
Now let $L$ be an indecomposable maximal Cohen--Macaulay 
$D$-module generated by $s$, $s^2t$, $t^3$ and $s(s^4-t^4)$. 
Put $M=\oplus_{k \in \mathbb{Z}} L_{3k+1}$. 
Then $M$ is a graded $A$-module of rank $2$. 
One can see that $M$ is generated by $s$, $s^2t^5$, $s^4t^3$ and $t^7$. 
In particular, $\widehat{M_{\fkm}}$ is a weakly special 
Cohen--Macaulay $\widehat{A_{\fkm}}$-module. 
We believe that $\widehat{M_{\fkm}}$ is indecomposable as 
$\widehat{A_{\fkm}}$-module. 
If it is true, then $\widehat{M_{\fkm}}$ is not special, as required. 
\par 
Michael Wemyss taught us that one could obtain from \cite[Example 1]{Wu} a weakly special Cohen--Macaulay $R$-module of rank $2$ which is not a special Cohen--Macaulay $R$-module. 
\end{rem}

\par 
Next, we introduce the notion of special ideals. 

\begin{defn}[\textbf{Special ideal}] \label{Spideal}
An $\fkm$-primary ideal $I \subset A$ is called a \textit{special ideal} 
if it is a good ideal (cf. Definition \ref{Ratgood}) 
and there exists a special Cohen--Macaulay $A$-module $M$ 
(equivalently, $M_j$ for some $j$) with respect to $I$. 
When this is the case, such a cycle $Z$ is called a \textit{special cycle}.  
\end{defn}

\par 
In the rest of this section, we give a characterization of special ideals 
in terms of cycles. 
Before doing that, we need the following lemma, which also plays an important role 
in Section \ref{Ulrich-nonGor}. 
\par 
Let $Z=\sum_{i=1}^r a_i E_i$ and $W = \sum_{i=1}^r b_i E_i$ be anti-nef cycles on $X$. 
Put $\inf(Z,W)= \sum_{i=1}^r \inf(a_i,b_i) E_i$, then one can easily see that 
$\inf(Z,W)$ is also an anti-nef cycle on $X$. 

\begin{lem} \label{FiltCycles}
Assume that $Z\ne Z_0$ is an anti-nef cycle on $X$. 
Then we can find the following anti-nef cycles $Z_1,\ldots,Z_s$ and 
positive cycles $Y_1,\ldots,Y_s$ so that 
$0 < Y_s \le Y_{s-1} \le \cdots \le Y_1 \le Z_0:$ 
\begin{equation}
\left\{
\begin{array}{rcl}
Z=Z_s &=& Z_{s-1} + Y_s, \\ 
Z_{s-1} & = & Z_{s-2} + Y_{s-1}, \\
&\vdots& \\
Z_2 & = & Z_1 + Y_2, \\
Z_1 & = & Z_0 + Y_1, 
\end{array}
\right.
\end{equation}
where $Z_0$ denotes the fundamental cycle on $X$. 
\end{lem}

\begin{proof}
We can take an integer $s \ge 1$ such that $Z \not \le s Z_0$ and 
$Z \le (s+1)Z_0$. 
Put $Z_k = \inf(Z, (k+1)Z_0)$ for every $k=1,\ldots,s$. 
Then $Z_1,\ldots,Z_s$ are anti-nef cycles.  
In particular, $Z_0 \le Z_1 \le Z_2 \le \cdots \le Z_s =Z$. 
Moreover, if we put $Y_{k} = Z_{k} - Z_{k-1}$ for every $k=1,\ldots,s$, then 
we can obtain the required sequence. 
\end{proof}

\par 
Under the notation as in Lemma \ref{FiltCycles}, we put 
$I_{k}=I_{Z_{k}}=H^0(X,\mathcal{O}_X(-Z_{k}))$ for 
every $k=0,1,\ldots,s$. Then each $I_{k}$ is a good ideal and  
\[
 I = I_s \subset I_{s-1} \subset \cdots \subset I_1 \subset I_0 = \fkm.
\]  

The following theorem is the main theorem in this section, which gives a criterion 
for $I=I_Z$ to be a special ideal
in terms of cycles. 

\begin{thm} \label{Spcycle}
Let $Z=\sum_{j=1}^r a_j E_j \ne Z_0$ be an anti-nef cycle on the minimal resolution $X \to \Spec A$, 
and put $I=I_Z$. 
Let $Z_0=\sum_{j=1}^r n_j E_j$ denote the fundamental cycle on $X$. 
Suppose that $1 \le i \le r$. 
Then the following conditions are equivalent$:$ 
\begin{enumerate}
 \item[$(1)$]  $M_i$ is a special Cohen--Macaulay $A$-module with respect to $I$. 
 \item[$(2)$]  $a_i= n_i \cdot \ell_A(A/I) $. 
 \item[$(3)$]  There exist positive cycles $0 < Y_s \le \ldots \le Y_1 \le Z_0$ 
and anti-nef cycles $Z_1,\ldots,Z_s$ 
so that $Z_k = Z_{k-1}+Y_{k}$ for each $k = 1,\ldots,s$ and 
\[
Z_{k-1}\cdot Y_{k} =0, \quad p_a(Y_{k})=0 \quad  
\text{and} \quad \coeff_{E_i} Y_{k} =n_i
\quad \text{for every $k =1,2,\ldots,s$}, 
\]
where $\coeff_{E_i} W$ stands for the coefficient of $E_i$ in a cycle $W$. 
\end{enumerate}
When this is the case, $\ell_A(A/I)=s+1$ and every $I_{k} := I_{Z_{k}}$ is a special ideal. 
Moreover, for every $k = 1,2,\ldots,s$, we obtain that $\Supp(Y_{k})$ is connected,   
$\Supp(Y_{k}) \subset \cup \{E_j \subset \Supp(Y_{k-1}) \,|\,E_jZ_{k-1}=0\}$, 
and $Y_{k}$ is the fundamental cycle on $\Supp(Y_{k})$. 
\end{thm}

\begin{rem}
$Z_k:=Y_k+Z_{k-1}$ is not always anti-nef 
even if $Y_k$ is the fundamental cycle on $X$ which satisfies $p_a(Y_k)=0$ and 
$Y_kZ_{k-1}=0$.  
\end{rem}

\par 
Let us begin the proof of Theorem \ref{Spcycle}. 
The following formula is one of the main tools in this paper. 

\begin{lem}[\textbf{Kato's  Riemann--Roch formula; \cite{Ka}, \cite{WY}}] \label{RR}
Let $Z$ be an anti-nef cycle on the minimal resolution of singularities $X$, 
and put $I_Z=H^0(X,\mathcal{O}_X(-Z))$. 
Then for any maximal Cohen--Macaulay $A$-module $M$, we have 
\[
 \ell_A(M/I_ZM) = \rank_A M \cdot \ell_A(A/I_Z)+ c_1(\widetilde{M})Z.
\]
In particular, 
\[
 \ell_A(A/I_Z) = - \dfrac{Z^2+K_XZ}{2} = 1- p_a(Z). 
\]
\end{lem}

\par 
The next lemma easily follows from Lemma \ref{RR}. 

\begin{lem} \label{Filtlength}
Under the notation as in Theorem $\ref{Spcycle}$, we have 
\[
 \ell_A(A/I_{k})  = \ell_A(A/I_{k-1}) - Y_{k}Z_{k-1}+1-p_a(Y_{k}). 
\]
\end{lem}

\begin{proof}
By Lemma \ref{RR}, we have 
\begin{eqnarray*}
\ell_A(A/I_{k}) 
&=& 1-p_a(Z_{k}) =1-p_a(Z_{k-1}+Y_{k}) \\
&=& 1-p_a(Z_{k-1})-p_a(Y_{k})-Y_{k}Z_{k-1}+1 \\
&=& \ell_A(A/I_{k-1})-Y_{k}Z_{k-1}+1-p_a(Y_{k}),
\end{eqnarray*}
as required. 
\end{proof}

\par
The following lemma is a key lemma in the proof of the theorem.

\begin{lem} \label{Key-Sp}
Under the notation as in Theorem $\ref{Spcycle}$, we have 
\begin{enumerate}
 \item[$(1)$] $a_i  \le n_i \cdot \ell_A(A/I)$. 
 \item[$(2)$]  
 Equality holds in (1) if and only if $M_i$ is a 
special Cohen--Macaulay $A$-module with respect to $I$. 
\end{enumerate}
\end{lem}

\begin{proof} 
By Kato's Riemann--Roch formula, we have 
\[
 \ell_A(M_i/IM_i) = \rank_A M_i \cdot \ell_A(A/I) + c_1(\widetilde{M_i})\cdot Z
 = n_i \cdot \ell_A(A/I) + a_i.
\]
On the other hand, $\mu_A(M_i) = 2 n_i$ because 
$M_i$ is a special Cohen--Macaulay $A$-module (with respect to $\fkm$). 
Hence 
\[
 \ell_A(M_i/IM_i)  \le \mu_A(M_i)  \cdot \ell_A(A/I) = 2 n_i \cdot  \ell_A(A/I).
\]  
Therefore $a_i \le n_i \cdot  \ell_A(A/I) $ and equality holds true if and only if 
$M_i/IM_i$ is $A/I$-free, which means that $M_i$ is a special Cohen--Macaulay $A$-module with 
respect to $I$. 
\end{proof}

\begin{proof}[Proof of Theorem $\ref{Spcycle}$]
$(1) \Longleftrightarrow (2)$ follows from Lemma \ref{Key-Sp}. 
\par 
$(2) \Longleftrightarrow (3)$: 
We use induction on $s$. 
By Lemma \ref{Filtlength}, we have 
\begin{eqnarray*}
&& n_i \cdot \ell_A(A/I) - \coeff_{E_i} Z \\
&=& n_i \left\{\ell_A(A/I_{s-1}) -Y_sZ_{s-1}+1-p_a(Y_s) \right\} - \coeff_{E_i} Z_{s-1} - \coeff_{E_i} Y_s \\
& = & \left\{n_i \cdot \ell_A(A/I_{s-1})- \coeff_{E_i} Z_{s-1}  \right\} 
+ n_i (-Y_sZ_{s-1}) + n_i(-p_a(Y_s)) + \left\{n_i - \coeff_{E_i} Y_s\right\}. 
\end{eqnarray*}
By the induction hypothesis, $n_i \cdot \ell_A(A/I_{s-1})- \coeff_{E_i} Z_{s-1}  \ge 0$. 
Since $Z_{s-1}$ is anti-nef, $-Y_sZ_{s-1}\ge 0$. 
As $A$ is rational, $-p_a(Y_s) \ge 0$. 
By the choice of $Y_s$,  $n_i - \coeff_{E_i} Y_s \ge 0$. 
Hence we obtain the required inequality, and 
equality holds if and only if 
\[
n_i \cdot\ell_A(A/I_{s-1})= \coeff_{E_i} Z_{s-1}, \quad  
j_sZ_{s-1} = p_a(Y_s)=0 \quad \text{and} \quad 
n_i= \coeff_{E_i} Y_s. 
\]
Therefore the assertion follows from the induction hypothesis. 
\par
Now suppose that one of (1),(2),(3) holds. The induction hypothesis implies 
that $I_{s-1}$ is a special ideal with $\ell(A/I_{s-1})=s$ and 
$\Supp(Y_{k})$ is connected  and  
$\Supp(Y_{k}) \subset  \cup \{E_j \,|\,E_jZ_{k-1}=0\}$ 
and $Y_{k}$ is the fundamental cycle on $\Supp(Y_{k})$ for every $k =1,\ldots,s-1$. 
Then it follows from Lemma \ref{Filtlength} that 
$\ell(A/I_{s})=\ell(A/I_{s-1})+1 = s+1$. 
As $Y_{s}Z_{s-1}=0$ implies that $\Supp(Y_s) \subset \cup \{E_j \,|\,E_jZ_{s-1}=0\}$. 
Moreover, since $p_a(Y_{s})=0$, $Y_{s}$ must be connected. 
\par  
Let us show that $Y_{s}$ is the fundamental cycle on $\Supp(Y_{s})$. 
For each $E_{j} \subset \Supp(Y_s)$, 
$Y_sE_j = Y_s E_j + Z_{s-1}E_j = Z_s E_j \le 0$. 
Hence  $Y_s$ is anti-nef on  $\Supp(Y_s)$. 
If $Y_s$ is not the fundamental cycle on  $\Supp(Y_s)$, then 
there exist an $E_j \subset \Supp(Y_s)$ 
and an anti-nef cycle $Y_s'$ on  $\Supp(Y_s)$ so that 
$Y_s = Y_s' + E_j$. 
Then 
\[
 0 = p_a(Y_s) =p_a(Y_s')+p_a(E_j)+Y_s'E_j -1 \le p_a(Y_s') -1 \le -1. 
\]
This is a contradiction. 
\end{proof}

\section{Ulrich ideals and modules over rational double points} \label{RDP}

The goal of this section is to classify
Ulrich ideals of any two-dimensional Gorenstein rational singularity 
(rational double point) $A$ and determine all of the Ulrich $A$-modules 
with respect to those ideals. 

\par 
First we recall the definition of rational double points. 

\begin{defn}[\textbf{Rational double point}] \label{RDP-def}
Let $A$ be a two-dimensional complete Noetherian local ring with unique maximal ideal $\fkm$ 
containing an algebraically closed field $k$. 
Then $A$ is said to be a \textit{rational double point} if 
it is isomorphic to the hypersurface $k[[x,y,z]]/(f)$, where 
$f$ is one of the following polynomials$:$
\[
\begin{array}{cll}
 (A_n) & z^2 + x^2 + y^{n+1} & (n \ge 1), \\
 (D_n) & z^2 + x^2y + y^{n-1} & (n \ge 4), \\
 (E_6) & z^2+x^3+y^4, & \\ 
 (E_7) & z^2+x^3+xy^3, & \\
 (E_8) & z^2+x^3+y^5. & 
\end{array}
\]  
\end{defn}

Note that $A$ is a $2$-dimensional 
Gorenstein rational singularity (of characteristic $0$) 
if and only if the $\fkm$-adic completion $\widehat{A}$ is a 
rational double point in the above sense. 

\par
The following theorem is the first main result in this section. 
In the latter half of this section, we give  
the complete classification 
of Ulrich ideals and modules as an application of the theorem. 

\begin{thm}[\textrm{See also Theorem \ref{UvsWsp}}] \label{SPvsUlrich-RDP}
Assume that $A$ is a rational double point of dimension $2$, 
and let $I \subset A$ be a nonparameter $\fkm$-primary ideal. 
Then the following conditions are equivalent$:$
\begin{enumerate}
 \item[$(1)$]  $M$ is an Ulrich $A$-module with respect to $I$. 
 \item[$(2)$]  $M$ is a special Cohen--Macaulay $A$-module with respect to $I$. 
 \item[$(3)$]  $M$ is a weakly special Cohen--Macaulay $A$-module with respect to $I$. 
 \item[$(4)$]  $M/IM$ is $A/I$-free and $M$ has no free summands. 
\end{enumerate}
When this is the case, $I$ is an Ulrich ideal 
and $M^{*} \cong \Syz_A^1(M)$ is also an Ulrich $A$-module with respect to $I$. 
\end{thm}

\par
In what follows, we prove Theorem \ref{SPvsUlrich-RDP}.
We need several lemmata. 

\begin{lem} \label{Stable-known}
Assume that $A$ is a rational double point of dimension $2$,  
and let $I \subset A$ be an $\fkm$-primary ideal. 
Then $e_I^0(A) \le 2 \cdot \ell_A(A/I)$ holds true 
and equality holds if and only if 
$I$ is a good ideal.  
\end{lem}

\begin{proof}
The lemma is well-known but we give a proof here for the convenience of the reader. 
Let $\overline{I}$ denote the integral closure of $I$. Take a minimal reduction $Q$ of $I$. 
Then since $Q$ is also a minimal reduction of $\overline{I}$ and $\overline{I}^2=Q\overline{I}$, 
we have 
\[
 I \subset \overline{I} \subset Q \colon \overline{I} \subset Q \colon I.
\]
The Matlis duality theorem implies that 
\[
  e_I^0(A) = \ell_A(A/Q) = \ell_A(A/I)+\ell_A(I/Q) \le \ell_A(A/I)+\ell_A(Q\colon I/Q)=2 \cdot \ell_A(A/I),
\]
and equality holds if and only if $I=Q\colon I$, that is, $I$ is a good ideal. 
\end{proof}

\par
Almost all of the maximal Cohen--Macaulay $A$-modules over a hypersurface of multiplicity $2$ 
can be regarded as Ulrich modules in the classical sense. 

\begin{lem}[\textrm{cf. \cite[Corollary 1.4]{HKuh}}] \label{HKuh}
Let $A$ be a hypersurface local domain of $e_{\fkm}^0(A)=2$. 
Then every maximal Cohen--Macaulay $A$-module without free summands satisfies 
$\mu_A(M)=e_{\fkm}^0(M)=2 \cdot \rank_A M$, that is, $M$ is an Ulrich $A$-module 
with respect to $\fkm$. 
\end{lem}


\par 
Any two-dimensional rational double point 
$A$ can be regarded as an invariant subring $B^G$, 
where $B=k[[s,t]]$ is a formal power series ring over $k$, 
and $G$ is a finite subgroup of $\mathrm{SL}(2,k)$. 
Thus we can apply the so-called McKay correspondence, 
which is a special case of `special McKay correspondence' (see Section \ref{Sp}). 

\begin{lem}[\textbf{McKay Correspondence}] \label{McKay}
Let $A=B^G$ as above. Then$:$
\begin{enumerate}
 \item[$(1)$]  The ring $A$ is of finite CM-representation type. 
Let $\{M_i\}_{i=0}^r$ be the set of isomorphism classes of indecomposable 
maximal Cohen--Macaulay $A$-modules, where $M_0=A$. 
Then $B \cong \bigoplus_{i=0}^r M_i^{\oplus n_i}$, where $n_i = \rank_A M_i$.  
 \item[$(2)$]  The fundamental cycle is given by $Z_0=\sum_{j=1}^r n_j E_j$ so that 
if we choose indices suitably, then 
$c_1(\widetilde{M_i})E_j = \delta_{ij}$ for $1 \le i,j \le r$, where 
$c_1(*)$ denotes the Chern class and $\widetilde{M_i} = \varphi^{*}(M_i)/\text{torsion}$.
In particular, $M_i$ is a special Cohen--Macaulay $A$-module 
(with respect to $\frm$) for every $i=1,2,\ldots,r$. 
\end{enumerate}
\end{lem}

\par
We are now ready to prove Theorem  \ref{SPvsUlrich-RDP}.

\begin{proof}[Proof of Theorem \ref{SPvsUlrich-RDP}]
$(1) \Longrightarrow (2):$ 
Since $M$ is an Ulrich $A$-module with respect to $I$, it has no free summands
because no free module is an Ulrich $A$-module with respect to $I$. 
Thus $M$ is an Ulrich $A$-module  with respect to $\fkm$ by Lemma \ref{HKuh}
and it is also a special Cohen--Macaulay $A$-module with respect to $\fkm$ 
by Lemma \ref{McKay}. 
Hence $M$ is a special Cohen--Macaulay $A$-module with respect to $I$ 
because $M/IM$ is $A/I$-free. 

\par 
 $(2) \Longrightarrow (3):$ See Lemma \ref{WSpvsSp}. 

\par 
 $(3) \Longrightarrow (4):$ Trivial. 
 
\par 
 $(4) \Longrightarrow (1):$
By Lemma \ref{HKuh}, $M$ is a weakly special Cohen--Macaulay $A$-module with respect to $I$. 
Note that $e_I^0(A) \le 2 \cdot \ell_A(A/I)$ by Lemma \ref{Stable-known}. 
By a similar argument as in the proof of Theorem \ref{UvsWsp}, we have 
\[
\ell_A(M/IM)=e_I^0(M) \quad \text{and} \quad e_I^0(A)=2 \cdot \ell_A(A/I), 
\]
whence $M$ is an Ulrich $A$-module with respect to $I$ and $I$ is a good ideal. 
\par 
Since $A$ is a hypersurface local domain, Proposition \ref{UvsWspideal} implies that 
$I$ is an Ulrich ideal. In particular, $A/I$ is Gorenstein. 
Thus applying Theorem  \ref{Dual} yields $M^{*} \cong \Syz_A^1(M)$ is also an Ulrich $A$-module 
with respect to $I$. 
\end{proof}

\par
The corollary below follows from Proposition \ref{UvsWspideal} 
and Theorem \ref{SPvsUlrich-RDP}. 

\begin{cor}
Assume that $A$ is a rational double point of dimension $2$. 
Let $I$ be an $\fkm$-primary ideal. 
Then the following conditions are equivalent. 
\begin{enumerate}
 \item[$(1)$]  $I$ is an Ulrich ideal. 
 \item[$(2)$]  $I$ is a special ideal. 
 \item[$(3)$]  $I$ is a weakly special ideal. 
 \item[$(4)$] There exist an Ulrich $A$-module with respect to $I$. 
\end{enumerate}
\end{cor}

In the rest of this section, we classify all Ulrich ideals and Ulrich modules 
over rational double points of dimension $2$ 
using the results in the previous section. 

\par 
Let $\{M_i\}_{i=0}^r$ be the set of indecomposable 
maximal Cohen--Macaulay $A$-modules so that 
$M_0=A$ and 
$c_1(\widetilde{M_i})E_j = \delta_{ij}$ for all $1 \le i,j \le r$.  
\par
Now suppose that $M$ is  an Ulrich $A$-module with respect to $I$. 
Then $M$ is a finite direct sum of $M_1,\ldots,M_r$: 
\[
 M \cong M_1^{\oplus k_1} \oplus \cdots \oplus M_r^{\oplus k_r}. 
\] 
because $M$ has no free summands. 
Whenever $k_i > 0$, $M_i$ must be an Ulrich $A$-module with respect to $I$. 
Hence  it suffices to characterize $M_i$ that is an Ulrich $A$-module 
with respect to $I$.  
On the other hand, Theorem \ref{SPvsUlrich-RDP} implies that 
$I$ is an Ulrich ideal and whence $I$ is a special ideal. 
Thus those ideals $I$ (or cycles $Z$) are determined by Theorem \ref{Spcycle}.
Moreover, it is not difficult to determine all $M_i$ that is an Ulrich module 
with respect to $I_Z$ by Theorem \ref{SPvsUlrich-RDP}.  

\par 
Let $I$ be a good ideal of $A$ and let $Z$ be an anti-nef cycle 
on the minimal resolution $X$ such that $I\mathcal{O}_X = \mathcal{O}_X(-Z)$ and 
$I=\H^0(X,\mathcal{O}_X(-Z))$, that is, $I=I_Z$. 
Then we call $Z$ an \textit{Ulrich cycle} 
if $I$ is an Ulrich ideal. 
Note that $Z$ is an Ulrich cycle if and only if it is a special cycle. 
\par 
Now let us illustrate the main theorem by the following example. 
Let $Z=2E_1+3E_2+4E_3+3E_4+2E_5+2E_6$ be an Ulrich cycle of a rational 
double point $A=k[[x,y,z]]/(x^3+y^4+z^2)$, 
and put $I=\H^0(X,\mathcal{O}_X(-Z))$.  
Then since $Z$ is an anti-nef cycle   
on the minimal resolution $X \to \Spec A$ with support
in $E=\bigcup_{i=1}^6 E_i$, $Z$ can be described as follows$:$   

\par 
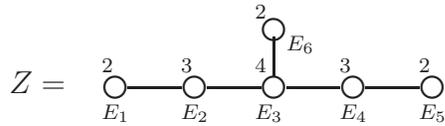
\vspace{6mm}
\begin{picture}(400,35)(-20,0)
    \thicklines
\put(-10,10){$Z=$}
  \put(25,18){{\tiny $2$}}
  \put(25,0){{\tiny $E_1$}}
\put(30,12){\circle{8}}
\put(35,12){\line(1,0){20}}
  \put(55,18){{\tiny $3$}}
  \put(55,0){{\tiny $E_2$}}
\put(60,12){\circle{8}}
\put(65,12){\line(1,0){20}}
  \put(83,18){{\tiny $4$}}
  \put(83,0){{\tiny $E_3$}}
\put(90,12){\circle{8}}
\put(95,12){\line(1,0){20}}
  \put(115,18){{\tiny $3$}}
  \put(115,0){{\tiny $E_4$}}
\put(120,12){\circle{8}}
\put(125,12){\line(1,0){20}}
  \put(145,18){{\tiny $2$}}
  \put(145,0){{\tiny $E_5$}}
\put(150,12){\circle{8}}
\put(90,16){\line(0,1){15}}
  \put(83,38){{\tiny $2$}}
  \put(95,26){{\tiny $E_6$}}
\put(90,34){\circle{8}}
\end{picture}

\par \vspace{2mm} \par \noindent
Furthermore, 
by Theorem \ref{Spcycle}(2)  $M_i$ is an Ulrich $A$-module with respect to $I$ if and only if 
$i=1$ or $5$ because $Z_0 = E_1+2E_2+3E_3+ 2E_4+E_5+2E_6$ and 
$\ell_A(A/I)=2$. 
In other words, any Ulrich $A$-module with respect to $I$ is given by 
$M \cong M_1^{\oplus a} \oplus M_5^{\oplus b}$ for some integers $a,b \ge 0$. 
We can describe this by the following picture. 
 
\par 
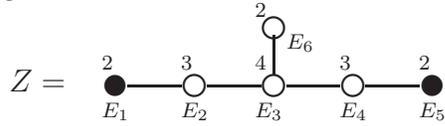
\vspace{3mm}
\begin{picture}(400,35)(-20,0)
    \thicklines
\put(-10,10){$Z=$}
  \put(25,18){{\tiny $2$}}
  \put(25,0){{\tiny $E_1$}}
\put(30,12){\circle*{8}}
\put(35,12){\line(1,0){20}}
  \put(55,18){{\tiny $3$}}
  \put(55,0){{\tiny $E_2$}}
\put(60,12){\circle{8}}
\put(65,12){\line(1,0){20}}
  \put(83,18){{\tiny $4$}}
  \put(83,0){{\tiny $E_3$}}
\put(90,12){\circle{8}}
\put(95,12){\line(1,0){20}}
  \put(115,18){{\tiny $3$}}
  \put(115,0){{\tiny $E_4$}}
\put(120,12){\circle{8}}
\put(125,12){\line(1,0){20}}
  \put(145,18){{\tiny $2$}}
  \put(145,0){{\tiny $E_5$}}
\put(150,12){\circle*{8}}
\put(90,16){\line(0,1){15}}
  \put(83,38){{\tiny $2$}}
  \put(95,26){{\tiny $E_6$}}
\put(90,34){\circle{8}}
\end{picture}

\par \vspace{4mm}
We are now ready to state the main theorem in this section. 

\begin{thm} \label{Main-RDP}
Let $A$ is a two-dimensional rational double point.  
Let $\varphi \colon X \to \Spec A$ be the minimal resolution of 
singularities with $E=\varphi^{-1}(\fkm) = \bigcup_{i=1}^r E_i$, the 
exceptional divisor on $X$. 
Then all Ulrich cycles $Z_k$ of $A$ and all  
indecomposable Ulrich $A$-modules with respect to 
$I_k=\H^0(X,\mathcal{O}_X(-Z_k))$ 
are given by the following$:$
\begin{flushleft}
$\bullet$ $(A_n)$ $x^2 + y^{n+1}+z^2$
\end{flushleft}
\par \noindent
When $n=2m$, the complete list of all Ulrich cycles 
is given by the following$:$ 
\par 
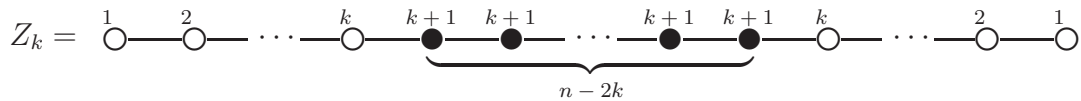
\vspace{4mm}
\begin{picture}(400,20)(-20,0)
    \thicklines
\put(-10,10){$Z_k=$}
\put(25,18){{\tiny $1$}}
\put(30,12){\circle{8}}
\put(35,12){\line(1,0){20}}
\put(55,18){{\tiny $2$}}
\put(60,12){\circle{8}}
\put(65,12){\line(1,0){15}}
\put(84,10){$\cdots$}
\put(100,12){\line(1,0){15}}
\put(115,18){{\tiny $k$}}
\put(120,12){\circle{8}}
\put(125,12){\line(1,0){20}}
\put(140,18){{\tiny $k+1$}}
\put(150,12){\circle*{8}}
\put(155,12){\line(1,0){20}}
\put(170,18){{\tiny $k+1$}}
\put(180,12){\circle*{8}}
\put(185,12){\line(1,0){15}}
\put(204,10){$\cdots$}
\put(220,12){\line(1,0){15}}
\put(230,18){{\tiny $k+1$}}
\put(240,12){\circle*{8}}
\put(245,12){\line(1,0){20}}
\put(260,18){{\tiny $k+1$}}
\put(270,12){\circle*{8}}
\put(275,12){\line(1,0){20}}
\put(295,18){{\tiny $k$}}
\put(300,12){\circle{8}}
\put(305,12){\line(1,0){15}}
\put(324,10){$\cdots$}
\put(340,12){\line(1,0){15}}
\put(355,18){{\tiny $2$}}
\put(360,12){\circle{8}}
\put(365,12){\line(1,0){20}}
\put(385,18){{\tiny $1$}}
\put(390,12){\circle{8}}
\put(148,6){$\underbrace{\phantom{aaaaaaaaaaaaaaaaaaaa}}$}
\put(198,-10){{\tiny $n-2k$}}
\end{picture}
\par \vspace{4mm}
for $k=0,1,\cdots,m-1(=\frac{n}{2}-1)$. 
Then $\ell_A(A/I_k)=k+1$ for each $k=0,1,\ldots,m-1$. 

\par \vspace{4mm}\par \noindent
When $n=2m+1$, the complete list of all Ulrich cycles 
is given by the following$:$ 
\par 
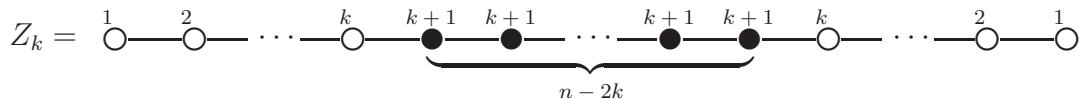
\vspace{4mm}
\begin{picture}(400,20)(-20,0)
    \thicklines
\put(-10,10){$Z_k=$}
\put(25,18){{\tiny $1$}}
\put(30,12){\circle{8}}
\put(35,12){\line(1,0){20}}
\put(55,18){{\tiny $2$}}
\put(60,12){\circle{8}}
\put(65,12){\line(1,0){15}}
\put(84,10){$\cdots$}
\put(100,12){\line(1,0){15}}
\put(115,18){{\tiny $k$}}
\put(120,12){\circle{8}}
\put(125,12){\line(1,0){20}}
\put(140,18){{\tiny $k+1$}}
\put(150,12){\circle*{8}}
\put(155,12){\line(1,0){20}}
\put(170,18){{\tiny $k+1$}}
\put(180,12){\circle*{8}}
\put(185,12){\line(1,0){15}}
\put(204,10){$\cdots$}
\put(220,12){\line(1,0){15}}
\put(230,18){{\tiny $k+1$}}
\put(240,12){\circle*{8}}
\put(245,12){\line(1,0){20}}
\put(260,18){{\tiny $k+1$}}
\put(270,12){\circle*{8}}
\put(275,12){\line(1,0){20}}
\put(295,18){{\tiny $k$}}
\put(300,12){\circle{8}}
\put(305,12){\line(1,0){15}}
\put(324,10){$\cdots$}
\put(340,12){\line(1,0){15}}
\put(355,18){{\tiny $2$}}
\put(360,12){\circle{8}}
\put(365,12){\line(1,0){20}}
\put(385,18){{\tiny $1$}}
\put(390,12){\circle{8}}
\put(148,6){$\underbrace{\phantom{aaaaaaaaaaaaaaaaaaaa}}$}
\put(198,-10){{\tiny $n-2k$}}
\end{picture}
\par \vspace{4mm}
for $k=0,1,\cdots,m(=\frac{n-1}{2})$. 
Then $\ell_A(A/I_k)=k+1$ for each $k=0,1,\ldots,m$. 
\vspace{4mm}
\begin{flushleft}
$\bullet$ $(D_n)$ $x^2y + y^{n-1}+z^2$ $(n \ge 4)$
\end{flushleft}
\par \noindent
When $n=2m$, the complete list of all Ulrich cycles 
is given by the following$:$ 
\par 
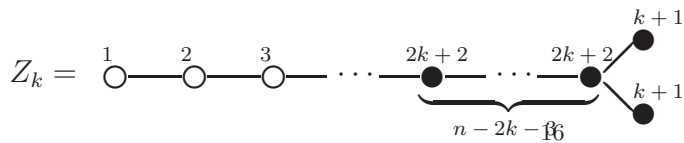
\vspace{5mm}
\begin{picture}(400,30)(-20,0)
    \thicklines
\put(-10,10){$Z_k=$}
  \put(25,18){{\tiny $1$}}
\put(30,12){\circle{8}}
\put(35,12){\line(1,0){20}}
  \put(55,18){{\tiny $2$}}
\put(60,12){\circle{8}}
\put(65,12){\line(1,0){20}}
  \put(85,18){{\tiny $3$}}
\put(90,12){\circle{8}}
\put(95,12){\line(1,0){15}}
\put(114,10){$\cdots$}
\put(130,12){\line(1,0){15}}
  \put(140,18){{\tiny $2k+2$}}
\put(150,12){\circle*{8}}
\put(155,12){\line(1,0){15}}
\put(174,10){$\cdots$}
\put(190,12){\line(1,0){15}}
  \put(195,18){{\tiny $2k+2$}}
\put(210,12){\circle*{8}}
\put(215,14){\line(1,1){11}}
\put(215,10){\line(1,-1){11}}
\put(230,26){\circle*{8}}
\put(230,-2){\circle*{8}}
  \put(226,32){{\tiny $k+1$}}
  \put(226,4){{\tiny $k+1$}}
\put(145,6){$\underbrace{\phantom{aaaaaaaaaaa}}$}
\put(158,-10){{\tiny $n-2k-3$}}
\end{picture}
\par \vspace{4mm}
for $k=0,1,\ldots,m-2(=\frac{n-4}{2})$. 
\par 
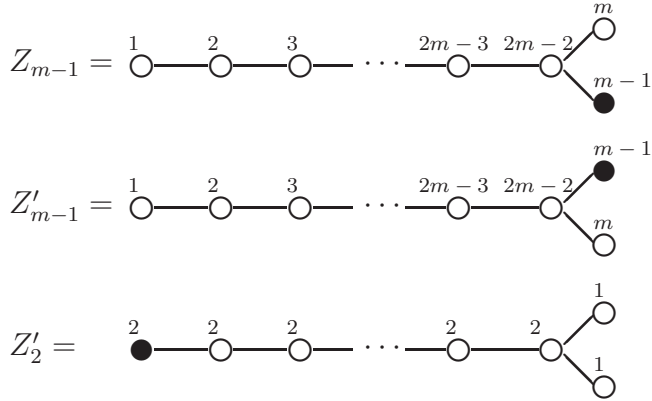
\vspace{8mm}
\begin{picture}(400,30)(-20,0)
    \thicklines
\put(-10,10){$Z_{m-1}=$}
  \put(35,18){{\tiny $1$}}
\put(40,12){\circle{8}}
\put(45,12){\line(1,0){20}}
  \put(65,18){{\tiny $2$}}
\put(70,12){\circle{8}}
\put(75,12){\line(1,0){20}}
  \put(95,18){{\tiny $3$}}
\put(100,12){\circle{8}}
\put(105,12){\line(1,0){15}}
\put(124,10){$\cdots$}
\put(140,12){\line(1,0){15}}
  \put(145,18){{\tiny $2m-3$}}
\put(160,12){\circle{8}}
\put(165,12){\line(1,0){25}}
  \put(177,18){{\tiny $2m-2$}}
\put(195,12){\circle{8}}
\put(200,14){\line(1,1){11}}
\put(200,10){\line(1,-1){11}}
\put(215,26){\circle{8}}
\put(215,-2){\circle*{8}}
  \put(211,32){{\tiny $m$}}
  \put(211,4){{\tiny $m-1$}}
\end{picture}
\par \vspace{8mm}
\begin{picture}(400,30)(-20,0)
    \thicklines
\put(-10,10){$Z_{m-1}'=$}
  \put(35,18){{\tiny $1$}}
\put(40,12){\circle{8}}
\put(45,12){\line(1,0){20}}
  \put(65,18){{\tiny $2$}}
\put(70,12){\circle{8}}
\put(75,12){\line(1,0){20}}
  \put(95,18){{\tiny $3$}}
\put(100,12){\circle{8}}
\put(105,12){\line(1,0){15}}
\put(124,10){$\cdots$}
\put(140,12){\line(1,0){15}}
  \put(145,18){{\tiny $2m-3$}}
\put(160,12){\circle{8}}
\put(165,12){\line(1,0){25}}
  \put(177,18){{\tiny $2m-2$}}
\put(195,12){\circle{8}}
\put(200,14){\line(1,1){11}}
\put(200,10){\line(1,-1){11}}
\put(215,26){\circle*{8}}
\put(215,-2){\circle{8}}
  \put(211,32){{\tiny $m-1$}}
  \put(211,4){{\tiny $m$}}
\end{picture}
\par \vspace{8mm}
\begin{picture}(400,30)(-20,0)
    \thicklines
\put(-10,10){$Z_2'=$}
  \put(35,18){{\tiny $2$}}
\put(40,12){\circle*{8}}
\put(45,12){\line(1,0){20}}
  \put(65,18){{\tiny $2$}}
\put(70,12){\circle{8}}
\put(75,12){\line(1,0){20}}
  \put(95,18){{\tiny $2$}}
\put(100,12){\circle{8}}
\put(105,12){\line(1,0){15}}
\put(124,10){$\cdots$}
\put(140,12){\line(1,0){15}}
  \put(155,18){{\tiny $2$}}
\put(160,12){\circle{8}}
\put(165,12){\line(1,0){25}}
  \put(187,18){{\tiny $2$}}
\put(195,12){\circle{8}}
\put(200,14){\line(1,1){11}}
\put(200,10){\line(1,-1){11}}
\put(215,26){\circle{8}}
\put(215,-2){\circle{8}}
  \put(211,32){{\tiny $1$}}
  \put(211,4){{\tiny $1$}}
\end{picture}
\par \vspace{4mm}  \par \noindent
Then $\ell_A(A/I_k)=k+1$ for each $k=0,1,\ldots,m-2$, 
$\ell_A(A/I_{m-1})=\ell_A(A/I_{m-1}')=m$ and   
$\ell_A(A/I_2')=2$, where $m=\frac{n}{2}$.

\par \vspace{5mm} 
\par \noindent
When $n=2m+1$, the complete list of all Ulrich cycles 
is given by the following$:$ 
\par 
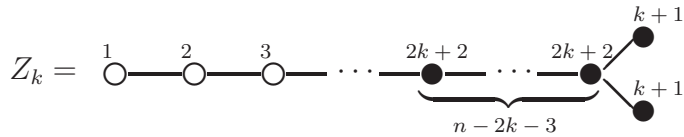
\vspace{5mm}
\begin{picture}(400,30)(-20,0)
    \thicklines
\put(-10,10){$Z_k=$}
  \put(25,18){{\tiny $1$}}
\put(30,12){\circle{8}}
\put(35,12){\line(1,0){20}}
  \put(55,18){{\tiny $2$}}
\put(60,12){\circle{8}}
\put(65,12){\line(1,0){20}}
  \put(85,18){{\tiny $3$}}
\put(90,12){\circle{8}}
\put(95,12){\line(1,0){15}}
\put(114,10){$\cdots$}
\put(130,12){\line(1,0){15}}
  \put(140,18){{\tiny $2k+2$}}
\put(150,12){\circle*{8}}
\put(155,12){\line(1,0){15}}
\put(174,10){$\cdots$}
\put(190,12){\line(1,0){15}}
  \put(195,18){{\tiny $2k+2$}}
\put(210,12){\circle*{8}}
\put(215,14){\line(1,1){11}}
\put(215,10){\line(1,-1){11}}
\put(230,26){\circle*{8}}
\put(230,-2){\circle*{8}}
  \put(226,32){{\tiny $k+1$}}
  \put(226,4){{\tiny $k+1$}}
\put(145,6){$\underbrace{\phantom{aaaaaaaaaaa}}$}
\put(158,-10){{\tiny $n-2k-3$}}
\end{picture}
\par \vspace{4mm}
for $k=0,1,\ldots,m-2(=\frac{n-5}{2})$. 

\par 
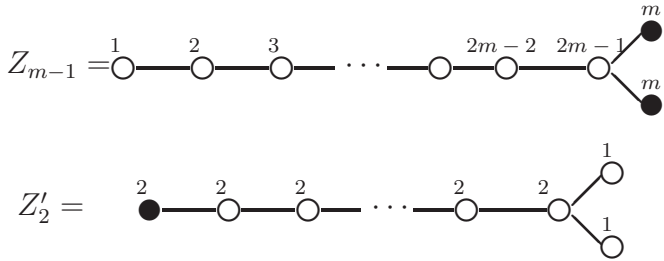
\vspace{5mm}
\begin{picture}(400,30)(-20,0)
    \thicklines
\put(-14,10){$Z_{m-1}=$}
  \put(25,18){{\tiny $1$}}
\put(30,12){\circle{8}}
\put(35,12){\line(1,0){20}}
  \put(55,18){{\tiny $2$}}
\put(60,12){\circle{8}}
\put(65,12){\line(1,0){20}}
  \put(85,18){{\tiny $3$}}
\put(90,12){\circle{8}}
\put(95,12){\line(1,0){15}}
\put(114,10){$\cdots$}
\put(130,12){\line(1,0){15}}
\put(150,12){\circle{8}}
\put(155,12){\line(1,0){15}}
\put(160,18){{\tiny $2m-2$}}
\put(175,12){\circle{8}}
\put(180,12){\line(1,0){25}}
  \put(194,18){{\tiny $2m-1$}}
\put(210,12){\circle{8}}
\put(215,14){\line(1,1){11}}
\put(215,10){\line(1,-1){11}}
\put(230,26){\circle*{8}}
\put(230,-2){\circle*{8}}
  \put(226,32){{\tiny $m$}}
  \put(226,4){{\tiny $m$}}
\end{picture}

\par \vspace{8mm}
\begin{picture}(400,30)(-20,0)
    \thicklines
\put(-10,10){$Z_{2}'=$}
  \put(35,18){{\tiny $2$}}
\put(40,12){\circle*{8}}
\put(45,12){\line(1,0){20}}
  \put(65,18){{\tiny $2$}}
\put(70,12){\circle{8}}
\put(75,12){\line(1,0){20}}
  \put(95,18){{\tiny $2$}}
\put(100,12){\circle{8}}
\put(105,12){\line(1,0){15}}
\put(124,10){$\cdots$}
\put(140,12){\line(1,0){15}}
  \put(155,18){{\tiny $2$}}
\put(160,12){\circle{8}}
\put(165,12){\line(1,0){25}}
  \put(187,18){{\tiny $2$}}
\put(195,12){\circle{8}}
\put(200,14){\line(1,1){11}}
\put(200,10){\line(1,-1){11}}
\put(215,26){\circle{8}}
\put(215,-2){\circle{8}}
  \put(211,32){{\tiny $1$}}
  \put(211,4){{\tiny $1$}}
\end{picture}
\par \vspace{4mm}  \par \noindent
Then $\ell_A(A/I_k)=k+1$ for each $k=0,1,\ldots,m-1$, 
and $\ell_A(A/I_2')=2$.

\vspace{4mm}
\begin{flushleft}
$\bullet$ $(E_6)$ $x^3 + y^4+z^2$ 
\end{flushleft}

\par \vspace{2mm}
The Ulrich cycles of $A$ are the following $Z_0$ and $Z_1$ with 
$\ell_A(A/I_{k})=k+1$ for each $k=0,1$:  
\par 
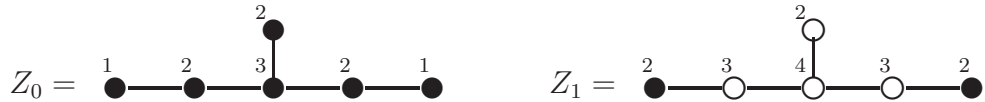
\vspace{5mm}
\begin{picture}(200,35)(-20,0)
    \thicklines
\put(-10,10){$Z_0=$}
  \put(25,18){{\tiny $1$}}
\put(30,12){\circle*{8}}
\put(35,12){\line(1,0){20}}
  \put(55,18){{\tiny $2$}}
\put(60,12){\circle*{8}}
\put(65,12){\line(1,0){20}}
  \put(83,18){{\tiny $3$}}
\put(90,12){\circle*{8}}
\put(95,12){\line(1,0){20}}
  \put(115,18){{\tiny $2$}}
\put(120,12){\circle*{8}}
\put(125,12){\line(1,0){20}}
  \put(145,18){{\tiny $1$}}
\put(150,12){\circle*{8}}
\put(90,16){\line(0,1){15}}
  \put(83,38){{\tiny $2$}}
\put(90,34){\circle*{8}}
\end{picture}
\begin{picture}(200,35)(-20,0)
    \thicklines
\put(-10,10){$Z_1=$}
  \put(25,18){{\tiny $2$}}
\put(30,12){\circle*{8}}
\put(35,12){\line(1,0){20}}
  \put(55,18){{\tiny $3$}}
\put(60,12){\circle{8}}
\put(65,12){\line(1,0){20}}
  \put(83,18){{\tiny $4$}}
\put(90,12){\circle{8}}
\put(95,12){\line(1,0){20}}
  \put(115,18){{\tiny $3$}}
\put(120,12){\circle{8}}
\put(125,12){\line(1,0){20}}
  \put(145,18){{\tiny $2$}}
\put(150,12){\circle*{8}}
\put(90,16){\line(0,1){15}}
  \put(83,38){{\tiny $2$}}
\put(90,34){\circle{8}}
\end{picture}
\par \vspace{4mm}  \par \noindent

\vspace{5mm}
\begin{flushleft}
$\bullet$ $(E_7)$ $x^3 + xy^3+z^2$ 
\end{flushleft}

\par \vspace{2mm}
The Ulrich cycles of $A$ are the following $Z_0$, $Z_1$ and $Z_2$ 
with $\ell_A(A/I_{k})=k+1$ for each $k=0,1,2$.  
\par 
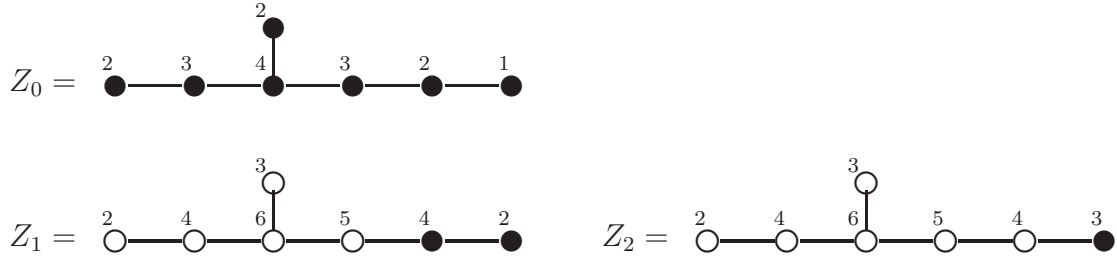
\vspace{5mm}
\begin{picture}(300,35)(-20,0)
    \thicklines
\put(-10,10){$Z_0=$}
  \put(25,18){{\tiny $2$}}
\put(30,12){\circle*{8}}
\put(35,12){\line(1,0){20}}
  \put(55,18){{\tiny $3$}}
\put(60,12){\circle*{8}}
\put(65,12){\line(1,0){20}}
  \put(83,18){{\tiny $4$}}
\put(90,12){\circle*{8}}
\put(95,12){\line(1,0){20}}
  \put(115,18){{\tiny $3$}}
\put(120,12){\circle*{8}}
\put(125,12){\line(1,0){20}}
  \put(145,18){{\tiny $2$}}
\put(150,12){\circle*{8}}
\put(155,12){\line(1,0){20}}
  \put(175,18){{\tiny $1$}}
\put(180,12){\circle*{8}}
\put(90,16){\line(0,1){15}}
  \put(83,38){{\tiny $2$}}
\put(90,34){\circle*{8}}
\end{picture}

\par \vspace{8mm}
\begin{picture}(200,35)(-20,0)
    \thicklines
\put(-10,10){$Z_1=$}
  \put(25,18){{\tiny $2$}}
\put(30,12){\circle{8}}
\put(35,12){\line(1,0){20}}
  \put(55,18){{\tiny $4$}}
\put(60,12){\circle{8}}
\put(65,12){\line(1,0){20}}
  \put(83,18){{\tiny $6$}}
\put(90,12){\circle{8}}
\put(95,12){\line(1,0){20}}
  \put(115,18){{\tiny $5$}}
\put(120,12){\circle{8}}
\put(125,12){\line(1,0){20}}
  \put(145,18){{\tiny $4$}}
\put(150,12){\circle*{8}}
\put(155,12){\line(1,0){20}}
  \put(175,18){{\tiny $2$}}
\put(180,12){\circle*{8}}
\put(90,16){\line(0,1){15}}
  \put(83,38){{\tiny $3$}}
\put(90,34){\circle{8}}
\end{picture}
\begin{picture}(200,35)(-40,0)
    \thicklines
\put(-10,10){$Z_2=$}
  \put(25,18){{\tiny $2$}}
\put(30,12){\circle{8}}
\put(35,12){\line(1,0){20}}
  \put(55,18){{\tiny $4$}}
\put(60,12){\circle{8}}
\put(65,12){\line(1,0){20}}
  \put(83,18){{\tiny $6$}}
\put(90,12){\circle{8}}
\put(95,12){\line(1,0){20}}
  \put(115,18){{\tiny $5$}}
\put(120,12){\circle{8}}
\put(125,12){\line(1,0){20}}
  \put(145,18){{\tiny $4$}}
\put(150,12){\circle{8}}
\put(155,12){\line(1,0){20}}
  \put(175,18){{\tiny $3$}}
\put(180,12){\circle*{8}}
\put(90,16){\line(0,1){15}}
  \put(83,38){{\tiny $3$}}
\put(90,34){\circle{8}}
\end{picture}

\vspace{5mm}
\begin{flushleft}
$\bullet$ $(E_8)$ $x^3 + y^5+z^2$ 
\end{flushleft}

\par \vspace{2mm}
The Ulrich cycles of $A$ are the following $Z_0$ and $Z_1$ 
with $\ell_A(A/I_{k})=k+1$ for each $k=0, 1$.
\par 
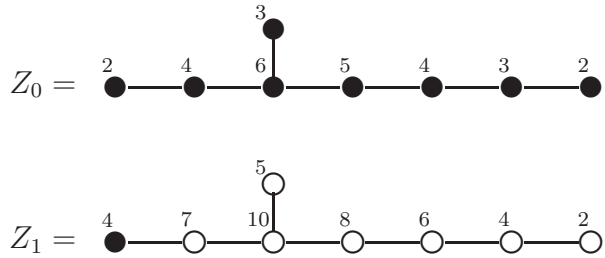
\vspace{5mm}
\begin{picture}(400,35)(-20,0)
    \thicklines
\put(-10,10){$Z_0=$}
  \put(25,18){{\tiny $2$}}
\put(30,12){\circle*{8}}
\put(35,12){\line(1,0){20}}
  \put(55,18){{\tiny $4$}}
\put(60,12){\circle*{8}}
\put(65,12){\line(1,0){20}}
  \put(83,18){{\tiny $6$}}
\put(90,12){\circle*{8}}
\put(95,12){\line(1,0){20}}
  \put(115,18){{\tiny $5$}}
\put(120,12){\circle*{8}}
\put(125,12){\line(1,0){20}}
  \put(145,18){{\tiny $4$}}
\put(150,12){\circle*{8}}
\put(155,12){\line(1,0){20}}
  \put(175,18){{\tiny $3$}}
\put(180,12){\circle*{8}}
\put(185,12){\line(1,0){20}}
  \put(205,18){{\tiny $2$}}
\put(210,12){\circle*{8}}
\put(90,16){\line(0,1){15}}
  \put(83,38){{\tiny $3$}}
\put(90,34){\circle*{8}}
\end{picture}

\par \vspace{8mm}
\begin{picture}(400,35)(-20,0)
    \thicklines
\put(-10,10){$Z_1=$}
  \put(25,18){{\tiny $4$}}
\put(30,12){\circle*{8}}
\put(35,12){\line(1,0){20}}
  \put(55,18){{\tiny $7$}}
\put(60,12){\circle{8}}
\put(65,12){\line(1,0){20}}
  \put(80,18){{\tiny $10$}}
\put(90,12){\circle{8}}
\put(95,12){\line(1,0){20}}
  \put(115,18){{\tiny $8$}}
\put(120,12){\circle{8}}
\put(125,12){\line(1,0){20}}
  \put(145,18){{\tiny $6$}}
\put(150,12){\circle{8}}
\put(155,12){\line(1,0){20}}
  \put(175,18){{\tiny $4$}}
\put(180,12){\circle{8}}
\put(185,12){\line(1,0){20}}
  \put(205,18){{\tiny $2$}}
\put(210,12){\circle{8}}
\put(90,16){\line(0,1){15}}
  \put(83,38){{\tiny $5$}}
\put(90,34){\circle{8}}
\end{picture}
\end{thm}

\par 
In our previous paper \cite[Section 9]{GOTWY}, we 
gave a complete list of 
the nonparameter Ulrich ideals 
for one-dimensional simple singularities. 
We can also do it for two-dimensional simple singularities 
(rational double points).

\begin{cor} \label{RDP-Uideal}
With the same notation as in Theorem $\ref{Main-RDP}$, 
the set $\calX_A$ is equal to$:$
\begin{center}
\begin{tabular}{cl}
$(A_{2m})$ &  $\{(x,y,z),(x,y^2,z),\ldots,(x,y^m,z)\}$. 
 \\[2mm]
$(A_{2m+1})$ & $\{(x,y,z),(x,y^2,z),\ldots,(x,y^{m+1},z)\}$.  \\[2mm] 
$(D_{2m})$ & $\{(x,y,z),(x,y^2,z),\ldots,(x,y^{m-1},z),$  \\[1mm]
& ~\quad $(x+\sqrt{-1}y^{m-1},y^{m},z),(x-\sqrt{-1}y^{m-1},y^m,z), (x^2,y,z)\}$. \\[2mm]
$(D_{2m+1})$  & $\{(x,y,z),(x,y^2,z),\ldots,(x,y^{m},z),(x^2,y,z)\}$. \\[2mm]
$(E_6)$ & $\{(x,y,z),(x,y^2,z) \}$. \\[2mm]
$(E_7)$ & $\{(x,y,z),(x,y^2,z),(x,y^3,z)\}$.  \\[2mm]
$(E_8)$ & $\{(x,y,z),(x,y^2,z) \}$. 
\end{tabular}
\end{center}
\end{cor}

\begin{proof}
One can easily see that any ideal $I$ appearing in the corollary has the form $I=Q+(z)$, where $Q$ is a parameter ideal of $A$ and 
$I^2 =QI$, $\ell_A(A/Q)= 2 \cdot \ell_A(A/I)$ and $\mu(I)=3$. 
Hence those ideals $I$ are Ulrich. 
\par
On the other hand, Theorem \ref{Main-RDP} implies that 
$\sharp\calX_A=m$ (resp. $m+1$, $m+2$, $m+1$, $2$, $3$ ,$2$)
if $A$ is a rational double point of type 
$(A_{2m})$ (resp. $(A_{2m+1})$, $(D_{2m})$, $(D_{2m+1})$, 
$(E_6)$, $(E_7)$, $(E_8)$). 
Hence the set as above coincides with $\calX_A$, respectively.     
\end{proof}

\begin{proof}[Proof of Theorem \ref{Main-RDP}]
\par
We first consider the cases $(E_6)$, $(E_7)$, $(E_8)$. 
\par \vspace{3mm} 
\underline{\bf The case $(E_6): f=x^3+y^4+z^2$}. 
The fundamental cycle $Z_0$ on the minimal resolution is given by 
\par 
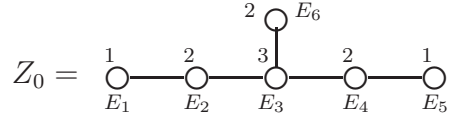
\vspace{5mm}
\begin{picture}(200,35)(0,0)
    \thicklines
\put(90,10){$Z_0=$}
  \put(125,18){{\tiny $1$}}
  \put(125,0){{\tiny $E_1$}}
\put(130,12){\circle{8}}
\put(135,12){\line(1,0){20}}
  \put(155,18){{\tiny $2$}}
  \put(155,0){{\tiny $E_2$}}
\put(160,12){\circle{8}}
\put(165,12){\line(1,0){20}}
  \put(183,18){{\tiny $3$}}
  \put(183,0){{\tiny $E_3$}}
\put(190,12){\circle{8}}
\put(195,12){\line(1,0){20}}
  \put(215,18){{\tiny $2$}}
    \put(215,0){{\tiny $E_4$}}
\put(220,12){\circle{8}}
\put(225,12){\line(1,0){20}}
  \put(245,18){{\tiny $1$}}
  \put(245,0){{\tiny $E_5$}}
\put(250,12){\circle{8}}
\put(190,16){\line(0,1){15}}
  \put(178,35){{\tiny $2$}}
  \put(197,35){{\tiny $E_6$}}  
\put(190,34){\circle{8}}
\end{picture}
\par \vspace{3mm}
Now suppose that $Z_0+Y$ is a special cycle for some positive cycle $Y \le Z_0$.  
Since $\cup \{E\,|\, EZ_0=0\} = \cup_{i=1}^5 E_i$ is connected, we have 
$Y=Y_1:=\sum_{i=1}^5 E_i$ in Theorem \ref{Spcycle}(3). 
Conversely, if we put
\par 
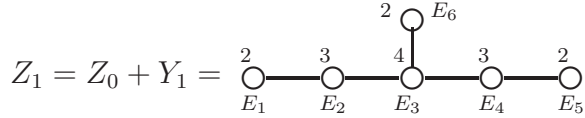
\vspace{5mm}
\begin{picture}(200,35)(0,0)
    \thicklines
\put(38,10){$Z_1=Z_0+Y_1=$}
  \put(125,18){{\tiny $2$}}
  \put(125,0){{\tiny $E_1$}}
\put(130,12){\circle{8}}
\put(135,12){\line(1,0){20}}
  \put(155,18){{\tiny $3$}}
  \put(155,0){{\tiny $E_2$}}
\put(160,12){\circle{8}}
\put(165,12){\line(1,0){20}}
  \put(183,18){{\tiny $4$}}
  \put(183,0){{\tiny $E_3$}}
\put(190,12){\circle{8}}
\put(195,12){\line(1,0){20}}
  \put(215,18){{\tiny $3$}}
    \put(215,0){{\tiny $E_4$}}
\put(220,12){\circle{8}}
\put(225,12){\line(1,0){20}}
  \put(245,18){{\tiny $2$}}
  \put(245,0){{\tiny $E_5$}}
\put(250,12){\circle{8}}
\put(190,16){\line(0,1){15}}
  \put(178,35){{\tiny $2$}}
  \put(197,35){{\tiny $E_6$}}  
\put(190,34){\circle{8}} .
\end{picture}
\par \vspace{3mm} \par \noindent
then $Z_1$ is anti-nef and $p(Y_1)=0$ because $Y_1$ can be regarded as the fundamental cycle 
on the dual graph of (the minimal resolution) of type $(A_5)$.
Hence  $Z_1$ is a special cycle 
and $M$ is an Ulrich $A$-module with respect to $I_{Z_1}$ if and only if 
it is a finite direct sum of $M_1$ and $M_5$ because 
$\coeff_{E_i} Y_1 = n_i(=1) \Longleftrightarrow i=1,5$; see Theorem \ref{Spcycle}.   
\par
Suppose that $Z_2 = Z_1+Y$ is a special cycle for some positive cycle $Y \le Y_1$. 
As $\cup \{E \subset \Supp(Y_1) \,|\, EZ_1=0\} = E_2 \cup E_3 \cup E_4$ is connected,  we have that 
$Y=E_2+E_3+E_4$ as the fundamental cycle on $\Supp(Y)=E_2 \cup E_3 \cup E_4$. 
But $Z_2 = Z_1+Y = 2E_1+4E_2+5E_3+4E_4+2E_5+2E_6$ is not anti-nef because$Z_2E_6=1$.  So 
the special cycles of $(E_6)$ are $Z_0$ and $Z_1$. 

\par \vspace{3mm}
\underline{\bf The case $(E_7): f=x^3+xy^3+z^2$}. 
The fundamental cycle $Z_0$ on the minimal resolution is given by 
\par 
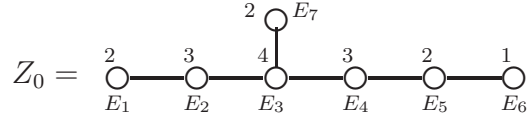
\vspace{5mm}
\begin{picture}(300,35)(0,0)
    \thicklines
\put(90,10){$Z_0=$}
  \put(125,18){{\tiny $2$}}
\put(125,0){{\tiny $E_1$}}  
\put(130,12){\circle{8}}
\put(135,12){\line(1,0){20}}
  \put(155,18){{\tiny $3$}}
  \put(155,0){{\tiny $E_2$}}
\put(160,12){\circle{8}}
\put(165,12){\line(1,0){20}}
  \put(183,18){{\tiny $4$}}
  \put(183,0){{\tiny $E_3$}}
\put(190,12){\circle{8}}
\put(195,12){\line(1,0){20}}
  \put(215,18){{\tiny $3$}}
  \put(215,0){{\tiny $E_4$}}
\put(220,12){\circle{8}} 
\put(225,12){\line(1,0){20}}
  \put(245,18){{\tiny $2$}}
  \put(245,0){{\tiny $E_5$}}
\put(250,12){\circle{8}}
\put(255,12){\line(1,0){20}}
  \put(275,18){{\tiny $1$}}
  \put(275,0){{\tiny $E_6$}}
\put(280,12){\circle{8}}
\put(190,16){\line(0,1){15}}
  \put(178,35){{\tiny $2$}}
  \put(196,35){{\tiny $E_7$}}
\put(190,34){\circle{8}}
\end{picture}
\par \vspace{3mm}
Since $\cup\{E\,|\, EZ_0=0\} = \cup_{i=2}^7 E_i$ is 
isomorphic to the dual graph of $(D_6)$, 
if $Z_1 = Z_0+Y$ is a special cycle for some positive cycle $Y \le Z_0$, then we have  
\par 
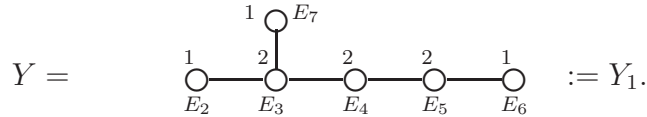
\vspace{5mm}
\begin{picture}(300,35)(0,0)
    \thicklines
\put(90,10){$Y=$}
  \put(155,18){{\tiny $1$}}
  \put(155,0){{\tiny $E_2$}}
\put(160,12){\circle{8}}
\put(165,12){\line(1,0){20}}
  \put(183,18){{\tiny $2$}}
  \put(183,0){{\tiny $E_3$}}
\put(190,12){\circle{8}}
\put(195,12){\line(1,0){20}}
  \put(215,18){{\tiny $2$}}
  \put(215,0){{\tiny $E_4$}}
\put(220,12){\circle{8}} 
\put(225,12){\line(1,0){20}}
  \put(245,18){{\tiny $2$}}
  \put(245,0){{\tiny $E_5$}}
\put(250,12){\circle{8}}
\put(255,12){\line(1,0){20}}
  \put(275,18){{\tiny $1$}}
  \put(275,0){{\tiny $E_6$}}
\put(280,12){\circle{8}}
\put(190,16){\line(0,1){15}}
  \put(178,35){{\tiny $1$}}
  \put(196,35){{\tiny $E_7$}}
\put(190,34){\circle{8}}
\put(300,10){$:=Y_1$.}
\end{picture}
\par \vspace{3mm} \par \noindent 
Conversely, one can easily see that 
the following $Z_1$ is a special cycle by Theorem \ref{Spcycle}. 
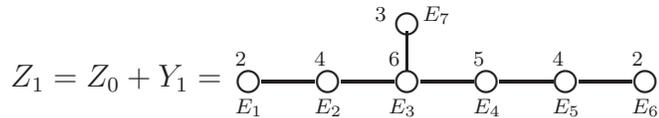
\par \vspace{5mm}
\begin{picture}(300,35)(0,0)
    \thicklines
\put(40,10){$Z_1=Z_0+Y_1=$}
  \put(125,18){{\tiny $2$}}
\put(125,0){{\tiny $E_1$}}  
\put(130,12){\circle{8}}
\put(135,12){\line(1,0){20}}
  \put(155,18){{\tiny $4$}}
  \put(155,0){{\tiny $E_2$}}
\put(160,12){\circle{8}}
\put(165,12){\line(1,0){20}}
  \put(183,18){{\tiny $6$}}
  \put(183,0){{\tiny $E_3$}}
\put(190,12){\circle{8}}
\put(195,12){\line(1,0){20}}
  \put(215,18){{\tiny $5$}}
  \put(215,0){{\tiny $E_4$}}
\put(220,12){\circle{8}} 
\put(225,12){\line(1,0){20}}
  \put(245,18){{\tiny $4$}}
  \put(245,0){{\tiny $E_5$}}
\put(250,12){\circle{8}}
\put(255,12){\line(1,0){20}}
  \put(275,18){{\tiny $2$}}
  \put(275,0){{\tiny $E_6$}}
\put(280,12){\circle{8}}
\put(190,16){\line(0,1){15}}
  \put(178,35){{\tiny $3$}}
  \put(196,35){{\tiny $E_7$}}
\put(190,34){\circle{8}}
\end{picture}
\par \vspace{3mm}
Note that $\cup \{E \subset \Supp(Y_1) \,|\,EZ_1=0\}$ admits 
two connected components:
\par 
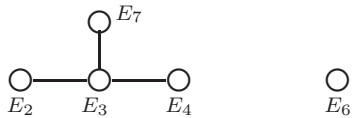
\vspace{5mm}
\begin{picture}(300,35)(-20,0)
    \thicklines
  \put(55,0){{\tiny $E_2$}}
\put(60,12){\circle{8}}
\put(65,12){\line(1,0){20}}
  \put(83,0){{\tiny $E_3$}}
\put(90,12){\circle{8}}
\put(95,12){\line(1,0){20}}
  \put(115,0){{\tiny $E_4$}}
\put(120,12){\circle{8}} 
  \put(175,0){{\tiny $E_6$}}
\put(180,12){\circle{8}}
\put(90,16){\line(0,1){15}}
  \put(96,35){{\tiny $E_7$}}
\put(90,34){\circle{8}}
\end{picture}
\par \vspace{3mm} \par \noindent 
The fundamental cycles $Y_2$ of their components are $E_2+2E_3+E_4+E_5$ and $E_6$, respectively. 
Note that $Z_1+(E_2+2E_3+E_4+E_5)$ is not anti-nef. 
So we are done. 

\par \vspace{3mm}
\underline{\bf The case $(E_8): x^3+y^5+z^2$}. 
The fundamental cycle $Z_0$ on the minimal resolution is given by 
\par \vspace{5mm}
\begin{picture}(400,35)(0,0)
    \thicklines
\put(90,10){$Z_0=$}
  \put(125,18){{\tiny $2$}}
  \put(125,0){{\tiny $E_1$}}
\put(130,12){\circle{8}}
\put(135,12){\line(1,0){20}}
  \put(155,18){{\tiny $4$}}
  \put(155,0){{\tiny $E_2$}}
\put(160,12){\circle{8}}
\put(165,12){\line(1,0){20}}
  \put(183,18){{\tiny $6$}}
  \put(183,0){{\tiny $E_3$}}
\put(190,12){\circle{8}}
\put(195,12){\line(1,0){20}}
  \put(215,18){{\tiny $5$}}
  \put(215,0){{\tiny $E_4$}}
\put(220,12){\circle{8}}
\put(225,12){\line(1,0){20}}
  \put(245,18){{\tiny $4$}}
  \put(245,0){{\tiny $E_5$}}
\put(250,12){\circle{8}}
\put(255,12){\line(1,0){20}}
  \put(275,18){{\tiny $3$}}
  \put(275,0){{\tiny $E_6$}}
\put(280,12){\circle{8}}
\put(285,12){\line(1,0){20}}
  \put(305,18){{\tiny $2$}}
   \put(305,0){{\tiny $E_7$}}
\put(310,12){\circle{8}}
\put(190,16){\line(0,1){15}}
  \put(178,35){{\tiny $3$}}
   \put(195,35){{\tiny $E_8$}}
\put(190,34){\circle{8}}
\end{picture}
\par \vspace{3mm}
Suppose that $Z_0+Y$ is a special cycle for some positive cycle $Y \le Z_0$. 
As $\cup \{E\,|\, EZ_0=0\} = \cup_{i\ne 7} E_i$ is connected and  the corresponding 
graph is isomorphic to the dual graph of $(E_7)$, we have 
\par \vspace{5mm}
\begin{picture}(400,35)(0,0)
    \thicklines
\put(90,10){$Y=$}
  \put(125,18){{\tiny $2$}}
  \put(125,0){{\tiny $E_1$}}
\put(130,12){\circle{8}}
\put(135,12){\line(1,0){20}}
  \put(155,18){{\tiny $3$}}
  \put(155,0){{\tiny $E_2$}}
\put(160,12){\circle{8}}
\put(165,12){\line(1,0){20}}
  \put(183,18){{\tiny $4$}}
  \put(183,0){{\tiny $E_3$}}
\put(190,12){\circle{8}}
\put(195,12){\line(1,0){20}}
  \put(215,18){{\tiny $3$}}
  \put(215,0){{\tiny $E_4$}}
\put(220,12){\circle{8}}
\put(225,12){\line(1,0){20}}
  \put(245,18){{\tiny $2$}}
  \put(245,0){{\tiny $E_5$}}
\put(250,12){\circle{8}}
\put(255,12){\line(1,0){20}}
  \put(275,18){{\tiny $1$}}
  \put(275,0){{\tiny $E_6$}}
\put(280,12){\circle{8}}
%
\put(190,16){\line(0,1){15}}
  \put(178,35){{\tiny $2$}}
   \put(195,35){{\tiny $E_8$}}
\put(190,34){\circle{8}}
\put(300,10){$:=Y_1$}
\end{picture}
\par \vspace{3mm}
\noindent
Conversely, if we put 
\par \vspace{5mm}
\begin{picture}(400,35)(0,0)
    \thicklines
\put(40,10){$Z_1=Z_0+Y_1=$}
  \put(125,18){{\tiny $4$}}
  \put(125,0){{\tiny $E_1$}}
\put(130,12){\circle{8}}
\put(135,12){\line(1,0){20}}
  \put(155,18){{\tiny $7$}}
  \put(155,0){{\tiny $E_2$}}
\put(160,12){\circle{8}}
\put(165,12){\line(1,0){20}}
  \put(181,18){{\tiny $10$}}
  \put(183,0){{\tiny $E_3$}}
\put(190,12){\circle{8}}
\put(195,12){\line(1,0){20}}
  \put(215,18){{\tiny $8$}}
  \put(215,0){{\tiny $E_4$}}
\put(220,12){\circle{8}}
\put(225,12){\line(1,0){20}}
  \put(245,18){{\tiny $6$}}
  \put(245,0){{\tiny $E_5$}}
\put(250,12){\circle{8}}
\put(255,12){\line(1,0){20}}
  \put(275,18){{\tiny $4$}}
  \put(275,0){{\tiny $E_6$}}
\put(280,12){\circle{8}}
\put(285,12){\line(1,0){20}}
  \put(305,18){{\tiny $2$}}
   \put(305,0){{\tiny $E_7$}}
\put(310,12){\circle{8}}
\put(190,16){\line(0,1){15}}
  \put(178,35){{\tiny $5$}}
   \put(195,35){{\tiny $E_8$}}
\put(190,34){\circle{8}}
\end{picture}
\par \vspace{3mm}\par \noindent
then $Z_1$ is a special cycle by Theorem \ref{Spcycle}.  
\par
Now suppose that $Z_1 + Y$ is a special cycle for some positive cycle $Y \le Y_1$. 
Since $\cup \{E \subset \Supp(Y_1) \,|\, EZ_1=0\}=E_3 \cup E_4 \cup E_5 \cup E_6 \cup E_8$ 
is connected, we have $Y = E_3+E_4+E_5+E_6+E_8$. 
But $Z_1 +(E_3+E_4+E_5+E_6+E_8)$ is not anti-nef. 

\par \vspace{3mm}
We next consider the case $(A_{n})$. 
\par \vspace{3mm}
\underline{\bf The case $(A_{2m})$: $f=x^2+y^{2m+1}+z^2$}.
The fundamental cycle $Z_0$ on the minimal resolution is given by 
\par \vspace{5mm}
\begin{picture}(400,20)(0,0)
    \thicklines
\put(90,10){$Z_0=$}
\put(125,18){{\tiny $1$}}
\put(124,0){{\tiny $E_1$}}
\put(130,12){\circle{8}}
\put(135,12){\line(1,0){20}}
\put(155,18){{\tiny $1$}}
\put(154,0){{\tiny $E_2$}}
\put(160,12){\circle{8}}
\put(165,12){\line(1,0){20}}
\put(185,18){{\tiny $1$}}
\put(184,0){{\tiny $E_3$}}
\put(190,12){\circle{8}}
\put(195,12){\line(1,0){15}}
\put(214,10){$\cdots$}
\put(230,12){\line(1,0){15}}
\put(241,18){{\tiny $1$}}
\put(238,0){{\tiny $E_{n-2}$}}
\put(250,12){\circle{8}}
\put(255,12){\line(1,0){20}}
\put(271,18){{\tiny $1$}}
\put(270,0){{\tiny $E_{n-1}$}}
\put(280,12){\circle{8}}
\put(285,12){\line(1,0){20}}
\put(301,18){{\tiny $1$}}
\put(298,0){{\tiny $E_{n}$}}
\put(310,12){\circle{8}}
\end{picture}
\par \vspace{3mm} \par \noindent
that is, $Z_0=\sum_{i=1}^n E_i$. 
\par
Now suppose that $Z_0+Y$ is a special cycle for some positive cycle $Y \le Z_0$. 
Since $\cup \{E\,|\, EZ_0=0\} = \cup_{i=2}^{n-1} E_i$ is connected, 
$Y=Y_1$, where $Y_1$ is the fundamental cycle on $ \cup_{i=2}^{n-1} E_i$,  
that is, $Y_1 = \sum_{i=2}^{n-1} E_i$. 
Conversely, 
\par \vspace{5mm}
\begin{picture}(400,20)(0,0)
    \thicklines
\put(40,10){$Z_1=Z_0+Y_1=$}
\put(125,18){{\tiny $1$}}
\put(124,0){{\tiny $E_1$}}
\put(130,12){\circle{8}}
\put(135,12){\line(1,0){20}}
\put(155,18){{\tiny $2$}}
\put(154,0){{\tiny $E_2$}}
\put(160,12){\circle{8}}
\put(165,12){\line(1,0){20}}
\put(185,18){{\tiny $2$}}
\put(184,0){{\tiny $E_3$}}
\put(190,12){\circle{8}}
\put(195,12){\line(1,0){15}}
\put(214,10){$\cdots$}
\put(230,12){\line(1,0){15}}
\put(241,18){{\tiny $2$}}
\put(238,0){{\tiny $E_{n-2}$}}
\put(250,12){\circle{8}}
\put(255,12){\line(1,0){20}}
\put(271,18){{\tiny $2$}}
\put(270,0){{\tiny $E_{n-1}$}}
\put(280,12){\circle{8}}
\put(285,12){\line(1,0){20}}
\put(301,18){{\tiny $1$}}
\put(298,0){{\tiny $E_{n}$}}
\put(310,12){\circle{8}}
\end{picture}
\par \vspace{3mm} \par \noindent
is a special cycle by Theorem \ref{Spcycle}. 
Similarly, if we put $Y_k=\sum_{i=k+1}^{2m-k} E_i$ for every $k=1,2,\ldots,m-1$, 
then we have 
\begin{enumerate}
\item[(a)] $0 < Y_{m-1} < Y_{m-2} < \ldots Y_2 < Y_1 \le Z_0$, 
\item[(b)] $p_a(Y_k)=0$ and $Y_kZ_{k-1}=0$ for every $k=1,\ldots,m-1$,
\item[(c)] $Z_k=Z_{k-1}+Y_k$ is anti-nef for every $k=1,\ldots,m-1$. 
\item[(d)] $\coeff_{E_i} Y_k = n_i$ if and only if $k+1 \le i \le 2m-k$. 
\end{enumerate}
This produces a sequence of Ulrich ideals:
\[
 I_{Z_{m-1}} \subset I_{Z_{m-2}} \subset \cdots \subset I_{Z_1} = \fkm.
\]
We can determine Ulrich ideals  in the case of $(A_{2m+1})$ similarly.
\par \vspace{3mm}
Finally, we consider the case $(D_{n}): f=x^2+xy^{n-3}+z^2$.
\par \vspace{3mm}
\underline{\bf The case $(D_{2m}): f=x^2+xy^{2m-3}+z^2$}. 
The fundamental cycle $Z_0$ on the minimal resolution of singularities 
is given by 
\par \vspace{5mm}
\begin{picture}(400,30)(0,0)
    \thicklines
\put(90,10){$Z_{0}=$}
  \put(135,18){{\tiny $1$}} 
  \put(135,0){{\tiny $E_1$}}
\put(140,12){\circle{8}}
\put(145,12){\line(1,0){20}}
  \put(165,18){{\tiny $2$}}
  \put(165,0){{\tiny $E_2$}}
\put(170,12){\circle{8}}
\put(175,12){\line(1,0){20}}
  \put(195,18){{\tiny $2$}}
   \put(195,0){{\tiny $E_3$}}
\put(200,12){\circle{8}}
\put(205,12){\line(1,0){15}}
\put(224,10){$\cdots$}
\put(240,12){\line(1,0){15}}
  \put(253,18){{\tiny $2$}}
  \put(241,0){{\tiny $E_{2m-3}$}}
\put(260,12){\circle{8}}
\put(265,12){\line(1,0){25}}
  \put(286,18){{\tiny $2$}}
  \put(278,0){{\tiny $E_{2m-2}$}}
\put(295,12){\circle{8}}
\put(300,14){\line(1,1){11}}
\put(300,10){\line(1,-1){11}}
\put(315,26){\circle{8}}
\put(315,-2){\circle{8}}
  \put(311,32){{\tiny $1$}}
  \put(320,23){{\tiny $E_{2m-1}$}}
  \put(311,4){{\tiny $1$}}
  \put(320,-5){{\tiny $E_{2m}$}}
\end{picture}
\par \vspace{3mm} \par \noindent
That is, $Z_0=E_1+2\sum_{i=2}^{2m-2} E_i + E_{2m-1}+E_{2m}$. 
\par 
Now suppose that $Z_0+Y$ is a special cycle on $X$ for some positive 
cycle $Y \le Z_0$. 
Since $\cup \{E\,|\,EZ_0=0\}$ has two connected components, 
we have that  $Y=E_1$ or $Y=Y_1$:
\par \vspace{5mm}
\begin{picture}(400,30)(-20,0)
    \thicklines
\put(10,10){$Y_{1}=$}
  \put(65,18){{\tiny $1$}}
  \put(65,0){{\tiny $E_3$}}
\put(70,12){\circle{8}}
\put(75,12){\line(1,0){20}}
  \put(95,18){{\tiny $2$}}
   \put(95,0){{\tiny $E_4$}}
\put(100,12){\circle{8}}
\put(105,12){\line(1,0){15}}
\put(124,10){$\cdots$}
\put(140,12){\line(1,0){15}}
  \put(153,18){{\tiny $2$}}
  \put(141,0){{\tiny $E_{2m-3}$}}
\put(160,12){\circle{8}}
\put(165,12){\line(1,0){25}}
  \put(186,18){{\tiny $2$}}
  \put(178,0){{\tiny $E_{2m-2}$}}
\put(195,12){\circle{8}}
\put(200,14){\line(1,1){11}}
\put(200,10){\line(1,-1){11}}
\put(215,26){\circle{8}}
\put(215,-2){\circle{8}}
  \put(211,32){{\tiny $1$}}
  \put(220,23){{\tiny $E_{2m-1}$}}
  \put(211,4){{\tiny $1$}}
  \put(220,-5){{\tiny $E_{2m}$}}
\end{picture} 
\par \vspace{3mm}
\noindent
Conversely, 
\par \vspace{5mm}
\begin{picture}(400,30)(0,0)
    \thicklines
\put(40,10){$Z_{1}'=Z_0+E_1=$}
  \put(135,18){{\tiny $2$}} 
  \put(135,0){{\tiny $E_1$}}
\put(140,12){\circle{8}}
\put(145,12){\line(1,0){20}}
  \put(165,18){{\tiny $2$}}
  \put(165,0){{\tiny $E_2$}}
\put(170,12){\circle{8}}
\put(175,12){\line(1,0){20}}
  \put(195,18){{\tiny $2$}}
   \put(195,0){{\tiny $E_3$}}
\put(200,12){\circle{8}}
\put(205,12){\line(1,0){15}}
\put(224,10){$\cdots$}
\put(240,12){\line(1,0){15}}
  \put(253,18){{\tiny $2$}}
  \put(241,0){{\tiny $E_{2m-3}$}}
\put(260,12){\circle{8}}
\put(265,12){\line(1,0){25}}
  \put(286,18){{\tiny $2$}}
  \put(278,0){{\tiny $E_{2m-2}$}}
\put(295,12){\circle{8}}
\put(300,14){\line(1,1){11}}
\put(300,10){\line(1,-1){11}}
\put(315,26){\circle{8}}
\put(315,-2){\circle{8}}
  \put(311,32){{\tiny $1$}}
  \put(320,23){{\tiny $E_{2m-1}$}}
  \put(311,4){{\tiny $1$}}
  \put(320,-5){{\tiny $E_{2m}$}}
\end{picture}
\par \vspace{3mm}
\noindent
and 
\par \vspace{5mm}
\begin{picture}(400,30)(0,0)
    \thicklines
\put(40,10){$Z_{1}=Z_0+Y_1=$}
  \put(135,18){{\tiny $1$}} 
  \put(135,0){{\tiny $E_1$}}
\put(140,12){\circle{8}}
\put(145,12){\line(1,0){20}}
  \put(165,18){{\tiny $2$}}
  \put(165,0){{\tiny $E_2$}}
\put(170,12){\circle{8}}
\put(175,12){\line(1,0){20}}
  \put(195,18){{\tiny $3$}}
   \put(195,0){{\tiny $E_3$}}
\put(200,12){\circle{8}}
\put(205,12){\line(1,0){20}}
  \put(225,18){{\tiny $4$}}
   \put(225,0){{\tiny $E_4$}}
\put(230,12){\circle{8}}
\put(235,12){\line(1,0){15}}
\put(254,10){$\cdots$}
\put(270,12){\line(1,0){15}}
  \put(283,18){{\tiny $4$}}
  \put(271,0){{\tiny $E_{2m-3}$}}
\put(290,12){\circle{8}}
\put(295,12){\line(1,0){25}}
  \put(316,18){{\tiny $4$}}
  \put(308,0){{\tiny $E_{2m-2}$}}
\put(325,12){\circle{8}}
\put(330,14){\line(1,1){11}}
\put(330,10){\line(1,-1){11}}
\put(345,26){\circle{8}}
\put(345,-2){\circle{8}}
  \put(341,32){{\tiny $2$}}
  \put(350,23){{\tiny $E_{2m-1}$}}
  \put(341,4){{\tiny $2$}}
  \put(350,-5){{\tiny $E_{2m}$}}
\end{picture}
\par \vspace{3mm} \par \noindent
are special cycles by Theorem \ref{Spcycle}. 
\par
Suppose that $Z_1+Y$ is a special cycle for some positive cycle $Y \le Y_1$. 
Since $\cup \{E \subset \Supp(Y_1) \,|\, EZ_1 =0\}$ has two connected components, 
we have $Y=E_3$ or $Y=Y_2$, where 
\par \vspace{5mm}
\begin{picture}(400,30)(0,0)
    \thicklines
\put(90,10){$Y_2=$}
  \put(165,18){{\tiny $1$}}
  \put(165,0){{\tiny $E_5$}}
\put(170,12){\circle{8}}
\put(175,12){\line(1,0){20}}
  \put(195,18){{\tiny $2$}}
   \put(195,0){{\tiny $E_6$}}
\put(200,12){\circle{8}}
\put(205,12){\line(1,0){15}}
\put(224,10){$\cdots$}
\put(240,12){\line(1,0){15}}
  \put(253,18){{\tiny $2$}}
  \put(241,0){{\tiny $E_{2m-3}$}}
\put(260,12){\circle{8}}
\put(265,12){\line(1,0){25}}
  \put(286,18){{\tiny $2$}}
  \put(278,0){{\tiny $E_{2m-2}$}}
\put(295,12){\circle{8}}
\put(300,14){\line(1,1){11}}
\put(300,10){\line(1,-1){11}}
\put(315,26){\circle{8}}
\put(315,-2){\circle{8}}
  \put(311,32){{\tiny $1$}}
  \put(320,23){{\tiny $E_{2m-1}$}}
  \put(311,4){{\tiny $1$}}
  \put(320,-5){{\tiny $E_{2m}$}}
\end{picture}
\par \vspace{5mm}
\noindent
Then $Z_1+E_3$ is not anti-nef, but 
\par \vspace{5mm}
\begin{picture}(400,30)(0,0)
    \thicklines
\put(90,10){$Z_2=$}
  \put(135,18){{\tiny $1$}} 
  \put(135,0){{\tiny $E_1$}}
\put(140,12){\circle{8}}
\put(145,12){\line(1,0){20}}
  \put(165,18){{\tiny $2$}}
  \put(165,0){{\tiny $E_2$}}
\put(170,12){\circle{8}}
\put(175,12){\line(1,0){20}}
  \put(195,18){{\tiny $3$}}
   \put(195,0){{\tiny $E_3$}}
\put(200,12){\circle{8}}
\put(205,12){\line(1,0){20}}
  \put(225,18){{\tiny $4$}}
   \put(225,0){{\tiny $E_4$}}
\put(230,12){\circle{8}}
\put(235,12){\line(1,0){15}}
\put(254,10){$\cdots$}
\put(270,12){\line(1,0){15}}
  \put(283,18){{\tiny $6$}}
  \put(271,0){{\tiny $E_{2m-3}$}}
\put(290,12){\circle{8}}
\put(295,12){\line(1,0){25}}
  \put(316,18){{\tiny $6$}}
  \put(308,0){{\tiny $E_{2m-2}$}}
\put(325,12){\circle{8}}
\put(330,14){\line(1,1){11}}
\put(330,10){\line(1,-1){11}}
\put(345,26){\circle{8}}
\put(345,-2){\circle{8}}
  \put(341,32){{\tiny $2$}}
  \put(350,23){{\tiny $E_{2m-1}$}}
  \put(341,4){{\tiny $2$}}
  \put(350,-5){{\tiny $E_{2m}$}}
\end{picture}
\par \vspace{3mm} \par \noindent
is a special cycle. 
Similarly, if we put 
\[
Y_k = E_{2k+1}+2\sum_{i=2k+2}^{2m-2} E_i + E_{2m-1}+E_{2m},\qquad 
Z_k = Z_{k-1}+Y_k
\]
for each $k=1,2,\ldots,m-2$, then we have a sequence of positive cycles 
\[
 0 < Y_{m-2} < Y_{m-3} < \cdots < Y_1 \le Z_0. 
\]
By Theorem \ref{Spcycle}, $Z_0$, $Z_1,\ldots,Z_{m-2}$ are special cycles. 
Note that
\[
\cup\{E \subset \Supp(Y_{m-2}) \,|\, EZ_{m-2}=0\} 
= E_{2m-3} \cup E_{2m-1} \cup E_{2m}
\] 
and 
\par \vspace{5mm}
\begin{picture}(400,30)(0,0)
    \thicklines
\put(80,10){$Z_{m-2}=$}
  \put(135,18){{\tiny $1$}} 
  \put(135,0){{\tiny $E_1$}}
\put(140,12){\circle{8}}
\put(145,12){\line(1,0){20}}
  \put(165,18){{\tiny $2$}}
  \put(165,0){{\tiny $E_2$}}
\put(170,12){\circle{8}}
\put(175,12){\line(1,0){20}}
  \put(195,18){{\tiny $3$}}
   \put(195,0){{\tiny $E_3$}}
\put(200,12){\circle{8}}
\put(205,12){\line(1,0){15}}
\put(224,10){$\cdots$}
\put(240,12){\line(1,0){15}}
  \put(245,18){{\tiny $2m-3$}}
  \put(241,0){{\tiny $E_{2m-3}$}}
\put(260,12){\circle{8}}
\put(265,12){\line(1,0){25}}
  \put(278,18){{\tiny $2m-2$}}
  \put(278,0){{\tiny $E_{2m-2}$}}
\put(295,12){\circle{8}}
\put(300,14){\line(1,1){11}}
\put(300,10){\line(1,-1){11}}
\put(315,26){\circle{8}}
\put(315,-2){\circle{8}}
  \put(317,32){{\tiny $m-1$}}
  \put(320,23){{\tiny $E_{2m-1}$}}
  \put(317,4){{\tiny $m-1$}}
  \put(320,-5){{\tiny $E_{2m}$}}
\end{picture}
\par \vspace{3mm} 
\noindent
By a similar argument as above, we obtain two minimal special cycles:
\par \vspace{5mm}
\begin{picture}(400,30)(0,0)
    \thicklines
\put(80,10){$Z_{m-1}=$}
  \put(135,18){{\tiny $1$}} 
  \put(135,0){{\tiny $E_1$}}
\put(140,12){\circle{8}}
\put(145,12){\line(1,0){20}}
  \put(165,18){{\tiny $2$}}
  \put(165,0){{\tiny $E_2$}}
\put(170,12){\circle{8}}
\put(175,12){\line(1,0){20}}
  \put(195,18){{\tiny $3$}}
   \put(195,0){{\tiny $E_3$}}
\put(200,12){\circle{8}}
\put(205,12){\line(1,0){15}}
\put(224,10){$\cdots$}
\put(240,12){\line(1,0){15}}
  \put(245,18){{\tiny $2m-3$}}
  \put(241,0){{\tiny $E_{2m-3}$}}
\put(260,12){\circle{8}}
\put(265,12){\line(1,0){25}}
  \put(278,18){{\tiny $2m-2$}}
  \put(278,0){{\tiny $E_{2m-2}$}}
\put(295,12){\circle{8}}
\put(300,14){\line(1,1){11}}
\put(300,10){\line(1,-1){11}}
\put(315,26){\circle{8}}
\put(315,-2){\circle{8}}
  \put(320,32){{\tiny $m$}}
  \put(320,23){{\tiny $E_{2m-1}$}}
  \put(320,4){{\tiny $m-1$}}
  \put(320,-5){{\tiny $E_{2m}$}}
\end{picture}
\par \vspace{8mm}
\begin{picture}(400,30)(0,0)
    \thicklines
\put(80,10){$Z_{m-1}'=$}
  \put(135,18){{\tiny $1$}} 
  \put(135,0){{\tiny $E_1$}}
\put(140,12){\circle{8}}
\put(145,12){\line(1,0){20}}
  \put(165,18){{\tiny $2$}}
  \put(165,0){{\tiny $E_2$}}
\put(170,12){\circle{8}}
\put(175,12){\line(1,0){20}}
  \put(195,18){{\tiny $3$}}
   \put(195,0){{\tiny $E_3$}}
\put(200,12){\circle{8}}
\put(205,12){\line(1,0){15}}
\put(224,10){$\cdots$}
\put(240,12){\line(1,0){15}}
  \put(245,18){{\tiny $2m-3$}}
  \put(241,0){{\tiny $E_{2m-3}$}}
\put(260,12){\circle{8}}
\put(265,12){\line(1,0){25}}
  \put(278,18){{\tiny $2m-2$}}
  \put(278,0){{\tiny $E_{2m-2}$}}
\put(295,12){\circle{8}}
\put(300,14){\line(1,1){11}}
\put(300,10){\line(1,-1){11}}
\put(315,26){\circle{8}}
\put(315,-2){\circle{8}}
  \put(320,32){{\tiny $m-1$}}
  \put(320,23){{\tiny $E_{2m-1}$}}
  \put(320,4){{\tiny $m$}}
  \put(320,-5){{\tiny $E_{2m}$}}
\end{picture}
\par \vspace{5mm}
\underline{\bf The case of $(D_{2m+1})$}. 
The fundamental cycle $Z_0$ on the minimal resolution is given by 
\par \vspace{5mm}
\begin{picture}(400,30)(0,0)
    \thicklines
\put(80,10){$Z_0=$}
  \put(135,18){{\tiny $1$}} 
  \put(135,0){{\tiny $E_1$}}
\put(140,12){\circle{8}}
\put(145,12){\line(1,0){20}}
  \put(165,18){{\tiny $2$}}
  \put(165,0){{\tiny $E_2$}}
\put(170,12){\circle{8}}
\put(175,12){\line(1,0){20}}
  \put(195,18){{\tiny $2$}}
   \put(195,0){{\tiny $E_3$}}
\put(200,12){\circle{8}}
\put(205,12){\line(1,0){15}}
\put(224,10){$\cdots$}
\put(240,12){\line(1,0){15}}
  \put(253,18){{\tiny $2$}}
  \put(241,0){{\tiny $E_{2m-2}$}}
\put(260,12){\circle{8}}
\put(265,12){\line(1,0){25}}
  \put(285,18){{\tiny $2$}}
  \put(278,0){{\tiny $E_{2m-1}$}}
\put(295,12){\circle{8}}
\put(300,14){\line(1,1){11}}
\put(300,10){\line(1,-1){11}}
\put(315,26){\circle{8}}
\put(315,-2){\circle{8}}
  \put(322,32){{\tiny $1$}}
  \put(320,23){{\tiny $E_{2m}$}}
  \put(320,4){{\tiny $1$}}
  \put(320,-5){{\tiny $E_{2m+1}$}}
\end{picture}
\par \vspace{5mm}
If we put 
\[
Y_k = E_{2k+1} + 2 \sum_{i=2k+2}^{2m-1} E_i + E_{2m}+ E_{2m+1}
\]
\[
Z_k = Z_{k-1} + Y_k = \sum_{i=1}^{2k+1} iE_i + (2k+2)\sum_{i=2k+2}^{2m-1} E_i + (k+1)(E_{2m}+E_{2m+1})
\]
for each $k=1,\ldots,m-2$, then 
$0 < Y_{m-2} < \cdots < Y_2 < Y_1 \le Z_0$ are positive cycles and 
$Z_0,Z_1,\ldots,Z_{m-2}$ are special cycles. 
\par 
Now suppose that $Z_{m-2}+Y$ is a special cycle for some positive cycle $Y \le Y_{m-2}$. 
Since 
\[
\cup\{E \subset \Supp(Y_{m-2}) \,|\, EZ_{m-2}=0\} = E_{2m-1} \cup E_{2m} \cup E_{2m+1},
\]
is connected, we have that $Y=E_{2m+1} + E_{2m}+E_{2m+1}$. 
\par 
Set $Y_{m-1} = E_{2m+1} + E_{2m}+E_{2m+1}$. 
Conversely, $Z_{m-1} = Z_{m-2}+Y_{m-1}$ is a special cycle by Theorem \ref{Spcycle}. 
Note that $Z_{m-1}$ is the minimal one among those special cycles. 
\end{proof}

\section{Ulrich ideals of non-Gorenstein rational singularities} \label{Ulrich-nonGor}

In this section, we study Ulrich ideals of two-dimensional 
non-Gorenstein rational singularities. 
Notice that the maximal ideal $\fkm$ is always an 
Ulrich ideal of such a local ring.

\par 
We first show that any Ulrich ideal of a two-dimensional 
rational singularity is a \text{good} ideal. 
In order to obtain a characterzation of Ulrich ideals, we need 
the following definition. 

\par
Throughout this section, let $(A,\fkm)$ be a $2$-dimensional 
{\it non-Gorenstein} rational singularity 
and $\varphi \colon X \to \Spec A$ be 
the {\it minimal}  resolution of singularities. 

\begin{defn} \label{Uz}
Let $\widetilde{\varphi} \colon \widetilde{X} \to \Spec A$ 
be a resolution of singularities of $\Spec A$. 
Decompose $\widetilde\varphi$ as $\widetilde{\varphi} = \varphi \circ \pi$, where $\pi \colon \widetilde{X} \to X$. 
Let $\pi^{*}Z_0$ denote the pull-back of the fundamental cycle $Z_0$ 
on the minimal resoluition to $\widetilde{X}$. 
Then for any anti-nef cycle $Z$ on $\widetilde{X}$, 
we put 
\[
U(Z)= (\varphi^{*}Z_0 \cdot Z)(p_a(Z)-1) + Z^2, 
\] 
where $p_a(Z)$ denotes the virtual genus of $Z$; see the remark below.  
\end{defn}

\begin{thm} \label{Uideal-chara}
Let $(A,\fkm)$ be a two-dimensional rational singularity. 
Let $I$ be an $\fkm$-primary ideal with $\mu_A(I) > 2$. 
Then the following conditions are equivaelnt$:$
\begin{enumerate}
 \item[$(1)$]  $I$ is an Ulrich ideal. 
 \item[$(2)$]  $\e_I^0(A) = (\mu(I)-1)\cdot \ell_A(A/I)$. 
 \item[$(3)$]  $I$ is an integrally closed ideal represented on the minimal resolution 
 of singularities $\varphi \colon X \to \Spec A$ such that 
 $I\mathcal{O}_X=\mathcal{O}_X(-Z)$, $I=\H^0(X,\mathcal{O}_X(-Z))$ and 
 $U(Z)=0$. 
\end{enumerate}
\end{thm}

\begin{proof}
$(1)\Longleftrightarrow (2)$ follows from \cite[Lemma 2.3]{GOTWY}.
\par \vspace{2mm}
$(3) \Longrightarrow (2):$ 
Any integrally closed ideal $I$ in a two-dimensional rational singularity is stable.
Moreover, $U(Z)=0$ means that $\e_{I}^0(A) = (\mu(I)-1)\cdot \ell_A(A/I)$. 
Thus the assertion immediately follows from this. 
\par \vspace{2mm}
$(2) \Longrightarrow (3):$ 
Since $I$ is an Ulrich ideal by (1), we have that $I=Q\colon I$ for any minimal 
reduction $Q$ of $I$ by \cite[Corollary 2.6]{GOTWY}.  
Then as $\overline{I}^2 = Q\overline{I}$, we get 
$I \subseteq \overline{I} \subseteq Q \colon \overline{I} \subseteq Q \colon I$.   
Hence $I=\overline{I}$ is integrally closed. 
\par
Let $\widetilde{\varphi} \colon \widetilde{X} \to \Spec A$ be a resolution of singularities 
so that $I=\H^0(\widetilde{X},\mathcal{O}_{\widetilde{X}}(-Z))$ and 
$I \mathcal{O}_{\widetilde{X}} = \mathcal{O}_{\widetilde{X}}(-Z)$ is invertible 
for some anti-nef cycle $Z$ on $\widetilde{X}$. 
Then (2) implies that $U(Z)=0$. 
\par
Now suppose that $I$ is \textit{not} represented on the minimal resolution of singularities
$\varphi \colon X \to \Spec A$. 
Then there exists a contraction $\psi \colon \widetilde{X} \to X'$ 
of a $(-1)$-curve $E$ on $\widetilde{X}$ 
such that $I$ is not represented on $X'$.  
Consider the following commutative diagram: 
\par \vspace{4mm}
\begin{picture}(400,35)
    \thicklines
  \put(90,25){$\widetilde{X}$}
  \put(105,28){\vector(1,0){40}}
  \put(150,25){$X$}
  \put(102,20){\vector(1,-1){12}}
  \put(115,0){$X'$}
  \put(132,10){\vector(1,1){12}}
\put(98,10){$\psi$}
\put(120,32){$\pi$}
\put(142,10){$\pi'$}
\end{picture}

\par \vspace{4mm} \noindent
Then we may assume that $Z= \psi^{*}Z'+nE$ for some anti-nef cycle $Z'$ on $X'$ 
and an integer $n \ge 1$. 
Note that $\pi^{*}Z_0 \cdot E = \psi^{*}Z' \cdot E =0$; 
see e.g. \cite[Fact 7.7]{GIW}. 
Then 
\begin{eqnarray*}
U(Z)-U(Z') & = & \left(\pi^{*}Z_0 \cdot (\psi^{*}Z'+nE)\right)
(p_a(\psi^{*}Z')+p_a(nE)+\psi^{*}Z'\cdot nE -2)  \\
&& + (\psi^{*}Z'+nE)^2- (\pi^{*}Z_0\cdot \psi^{*}Z')\left(p_a(\psi^{*}Z')-1\right)-(\psi^{*}Z')^2 \\
&=& (\pi^{*}Z_0 \cdot \psi^{*}Z')(p_a(nE)-1)+(nE)^2 \\
&=& \left((\pi')^{*}Z_0 \cdot Z' + 2\right) \frac{(nE)^2}{2} 
+ \frac{n(K_X \cdot E)}{2}
\left((\pi')^{*}Z_0 \cdot Z' \right).
\end{eqnarray*} 
Since $(\pi')^{*}Z_0 \cdot Z' \le -2$ and $E^2 = K_X \cdot E = -1$, we get 
$U(Z) > U(Z') \ge 0$.
This is a contradiction.  
\end{proof}

In what follows, we always assume that $\varphi \colon X \to \Spec A$ be 
the minimal resolution of singularities and 
$I\mathcal{O}_X = \mathcal{O}_X(-Z)$ is invertible 
and $I = \H^0(X,\mathcal{O}_X(-Z))$ for some anti-nef cycle $Z$ on $X$. 
Let $\varphi^{-1}(\fkm) = \bigcup_i E_i$ denote the exceptional divisor on $X$ with the irreducible 
components $\{E_i\}_{1 \le i \le r}$. 
Let $Z_0$ (resp. $K$) denotes the fundamental cycle 
(resp. the canonical divisor) on $X$. 
Notice that $Z_0E \le 0$ and $KE = -E^2 - 2$ for all exceptional curves $E$.    
\par
The next target is to characterize Ulrich cycles in terms of dual graphs.  
In order to do that, we recall the sequence of anti-nef cycles introduced in Lemma \ref{FiltCycles}.
Assume that $Z\ne Z_0$ is an anti-nef cycle on $X$. 
Then we can find the following anti-nef cycles $Z_1,\ldots,Z_s$ and 
positive cycles $Y_1,\ldots,Y_s$ so that 
$0 < Y_s \le Y_{s-1} \le \cdots \le Y_1 \le Z_0$: 
\begin{equation}
\left\{
\begin{array}{rcl}
Z=Z_s &=& Z_{s-1} + Y_s, \\ 
Z_{s-1} & = & Z_{s-2} + Y_{s-1}, \\
&\vdots& \\
Z_2 & = & Z_1 + Y_2, \\
Z_1 & = & Z_0 + Y_1. 
\end{array}
\right.
\end{equation}

\par
The following lemma plays a key role in the proof of the main theorem in
this section. 

\begin{lem} \label{Key-Ul}
Let $Z$, $Z'$ be anti-nef cycles on $X$ with $Z' =  Z+Y$, 
where $Y$ is a positive cycle. 
Then$:$
\begin {enumerate}
\item[$(1)$]  
$U(Z') - U(Z) 
= (YZ_0)\big\{(p_a(Z)-1)+(p_a(Y)-1)\big\}
 +(YZ)(Z'Z_0+2)\\
\hspace{8cm}
+ (p_a(Y)-1)(ZZ_0+2)-KY$.
\item[$(2)$]  Assume that $0 \ne Y \le Z_0$ and $e=\e_{\fkm}^0(A) \ge 3$. 
Then $U(Z')\ge U(Z)$ holds true, and equality holds if and only if 
$YZ=YZ_0 = p_a(Y) = (Z-Z_0)Z_0 = K(Z_0-Y)=0$. 
\end{enumerate}
\end{lem}

\begin{proof}
Since $p_a(Z+Y) = p_a(Z)+p_a(Y)+YZ-1$ by definition, we have 
\begin{eqnarray*}
U(Z')-U(Z) & =& 
(ZZ_0+YZ_0) (p_a(Z)-1+p_a(Y)-1+YZ)+(Z^2+2YZ+Y^2)\\
&& - (ZZ_0)(p_a(Z)-1)-Z^2 \\
&=& (YZ_0)\big\{(p_a(Z)-1) + (p_a(Y)-1) \big\} + (YZ)(ZZ_0+YZ_0+2) \\
&& + (p_a(Y)-1)(ZZ_0) + Y^2 \\
& = & (YZ_0)\big\{(p_a(Z)-1) + (p_a(Y)-1) \big\} + (YZ)(Z'Z_0+2) \\
&&+ (p_a(Y)-1)(ZZ_0+2) - KY,  
\end{eqnarray*}
where the last equality follows from $2(p_a(Y)-1) = KY+Y^2$. 

\par \vspace{2mm}
(2) Assume that $Y \le Z_0$. As $X \to \Spec A$ is the minimal resolution, 
we have that $KY \le KZ_0$ because $KE \ge 0$ for all curves $E$ on $X$. 
Since $Z_0$ is anti-nef and $Z-Z_0$, $Y$ are positive, we get 
\[
 Z'Z_0 + 2 = (Z-Z_0)Z_0 + YZ_0 + (Z_0^2+2) \le Z_0^2 +2= -e+2 < 0.
\]
Moreover, $p_a(Z_0)=0$ implies that 
\[
(p_a(Y)-1)(ZZ_0+2)-KY 
= p_a(Y)(ZZ_0+2) -(Z-Z_0)Z_0 -K(Y-Z_0) \ge 0
\]
and equality holds if and only if $p_a(Y)=(Z-Z_0)Z_0=K(Y-Z_0)=0$. 
\par
Note that $YZ_0$, $YZ \le 0$ and $p_a(Z)-1 + p_a(Y)-1 < 0$. 
Hence $U(Z') \ge U(Z)$ and equality holds if and only if $YZ_0=YZ=0$ and 
$p_a(Y)=(Z-Z_0)Z_0=K(Y-Z_0)=0$. 
\end{proof}

\par 
The main result in this section is the following theorem, which enables us to 
determine all Ulrich ideals of a two-dimensional (non-Gorenstein) 
rational singularity. 
For a positve cycle $Z$ on $X$, we write $Z=\sum_{E} Z_E E$, 
where $Z_E$ is a nonnegative integer.  

\begin{thm} \label{All-Uideal}
Let $(A,\fkm)$ be a two-dimensional rational singularity with 
$e=\e_{\fkm}^0(A)\ge 3$, and 
let $\varphi \colon X \to \Spec A$ be the minimal resolution of singularities. 
Set $Z_0 = \sum_{E} n_E E$, the fundamental cycle on $X$. 
Let $Z$ be an anti-nef cycle on $X$ with 
$I\mathcal{O}_X=\mathcal{O}_X(-Z)$ and $I=\H^0(X,\mathcal{O}_X(-Z))$.  
Then the following conditions are equivalent$:$ 
\begin{enumerate}
 \item[$(1)$]  $I$ is an Ulrich ideal, that is, $Z$ is an Ulrich cycle on $X$. 
 \item[$(2)$]  There exist a sequence of anti-nef cycles $Z_1,\ldots,Z_s$ and 
  a sequence of positive cycles 
$0< Y_s \le \cdots \le Y_1 \le Z_0$ for some $s \ge 1$   
  so that 
  \begin{equation}
  \left\{
  \begin{array}{rcl}
   Z=Z_s &=& Z_{s-1} + Y_s, \\
    Z_{s-1} & =& Z_{s-2} + Y_{s-1}, \\
    & \vdots & \\
    Z_1 & = & Z_0 + Y_1.
  \end{array}
  \right.
  \end{equation}
and  $Y_kZ_{k-1} = p_a(Y_k) = K(Z_0-Y_k) =0$ for every $k=1,\ldots,s$. 
\end{enumerate}
\par
When this is the case, the following conditions are satisfied. 
  \begin{enumerate}
   \item[$(a)$]  $\{E \,|\, E^2 \le -3\}$ is contained in $\Supp(Y_1)$.  
   \item[$(b)$]  $\Supp(Y_k)$ is given as one of the connected components of 
$\{E\,|\, EZ_{k-1}=0\}$ in $\{E\,|\,EZ_0=0\}$. 
 \item[$(c)$]  $Y_k$ is the fundamental cycle on $\Supp(Y_k)$. 
 \item[$(d)$]  $\coeff_E Z_0 =\coeff_E Y_k$ for every $E$ with $E^2 \le -3$.  
 \end{enumerate}

If, in addition, 
we put $I_k=\H^0(X,\mathcal{O}_X(-Z_k))$, 
then $I_k$ is an Ulrich ideal so that 
\[
 \fkm = I_0 \supseteq I_1 \supseteq \cdots \supseteq I_s = I \quad 
\text{and} \quad \ell_A(A/I)=s+1.
\] 
\end{thm}

\begin{proof}
Take a sequence as in (2). 
\par \vspace{2mm} \par \noindent 
$(1) \Longrightarrow (2):$ 
Lemma \ref{Key-Ul} implies that 
\[
0 = U(Z) = U(Z_s) \ge U(Z_{s-1}) \ge \cdots \ge U(Z_1) \ge U(Z_0)=0. 
\]
Hence all $Z_k$ are Ulrich cycles and 
\[
Y_kZ_{k-1} = Y_kZ_0 = p_a(Y_k) = (Z_k-Z_0)Z_0 = K(Z_0-Y_k)=0
\]
for every $k=1,\ldots,s$. 
By a similar argument as in the proof of Theorem \ref{Spcycle}, we have $\ell_A(A/I)=s+1$. 
\par 
If $E^2 \le -3$, then $KE=-E^2-2 > 0$. 
Thus $K(Z_0-Y_1) =0$ implies that $\coeff_E Z_0 =  \coeff_E Y_1$ for every $E$ with 
$E^2 \le -3$. 
In particular, $\Supp(Y_k) \supseteq \{E\,|\, E^2 \le -3\}$. 
On the other hand, $Y_kZ_0=0$ implies that 
$\Supp(Y_i) \subseteq \{E\,|\, EZ_0=0\}$ because $Z_0$ is an anti-nef cycle. 
\par \vspace{2mm}
Now suppose (2). 
Fix $i$ with $1 \le k \le r$. 
Since $Z_{k-1}$ is anti-nef and $Y_kZ_{k-1}=0$, a similar argument to the above 
yields that $\{E \,|\, E^2 \le -3\} \subseteq \Supp(Y_k) \subseteq 
\{E \,|\,EZ_{k-1}=0\}$. 
\par 
As $p_a(Y_k)=0$, $\Supp(Y_k)$ is connected. 
Moreover, $\Supp(Y_k)$ is one of the connected components of 
$\{E\,|\,EZ_{k-1}=0\}$. 
Indeed, if there exists a curve $E \notin \Supp(Y_k)$ such that $EE'>0$
for some $E' \in \Supp(Y_k)$, then $EZ_{k-1}< 0$ since $EY_k \ge 1$ and $EZ_{k-1}+EY_k = EZ_k \le 0$.
\begin{description}
\item[Claim] $Y_k$ is the fundamental cycle on $\Supp(Y_k)$.  
\end{description}
Take $E \in \Supp(Y_k)$. 
As $Y_kZ_{k-1}=0$, we have $EZ_{k-1} =0$. 
If $EY_k > 0$, then $EY_k = EZ_k \le 0$. 
This is a contradiction. 
Hence $EY_k\le 0$. Namely, $Y_k$ is anti-nef. 
Moreover, if $Y_k$ is not the fundamental cycle on $\Supp(Y_k)$, 
then we know that $p_a(Y_k) \le -1$. 
This contradicts the assumption $p_a(Y_k)=0$. 
Hence $Y_k$ must be the fundamental cycle on $\Supp(Y_k)$. 
\par \vspace{2mm}
To see $(2) \Longrightarrow (1)$, we notice that 
$(c)$ implies that $p_a(Y_k)=0$. 
Condition $(b)$ means that $Y_kZ_{k-1} =0$. 
Hence $Y_kZ_0=0$. 
Note that the equalities $Y_kZ_0=Y_{k-1}Z_0=\cdots =Y_1Z_0=0$ yield $(Z_k-Z_0)Z_0=0$.  
Condition $(d)$ implies that $K(Z_0-Y_k)=0$. 
Therefore 
$U(Z)=U(Z_r)=\cdots = U(Z_1)=U(Z_0)=0$, 
as required. 
\end{proof}

The following assertion does not hold true without the assumption that $A$ is rational; 
see \cite[Example 2.2]{GOTWY}.

\begin{cor} \label{Gor-Ul}
Let $A$ be a two-dimensional rational singularity. 
If $I$ is an Ulrich ideal of $A$, then $I$ is a special ideal and $A/I$ is Gorenstein. 
\end{cor}

\begin{proof}
Denote by $\fkm$ the maximal ideal of $R$.
We may assume that $A$ is not Gorenstein, that is, $e=\e_\fkm^0(A) \ge 3$. 
Then by Theorem \ref{All-Uideal}, we can find a sequence of Ulrich cycles 
$Z_1,\ldots,Z_s$ and positive cycles $0< Y_1 \le \cdots \le Y_s \le Z_0$ satisfying all 
conditions in Theorem \ref{All-Uideal}
so that 
\begin{equation}
\left\{
\begin{array}{rcl}
Z=Z_s & = & Z_{s-1}+Y_s, \\
Z_{s-1} &=& Z_{s-2} + Y_{s-1}, \\
& \vdots & \\
Z_1 & = & Z_0 + Y_1. 
\end{array}
\right.
\end{equation}
Then $Z \le (s+1)Z_0$ and $Z \not \le s Z_0$. 
In particular, $\fkm^{s} \not \subseteq I$ and $\fkm^{s+1} \subseteq I$. 
Moreover, $I$ is a special ideal by Theorem \ref{Spcycle}.
We have only to show the following claim.
\begin{description}
\item[Claim] There exists a minimal set of 
generators $\{u_1,\ldots,u_p,t\}$ such that 
$I=(u_1,\ldots,u_p,t^{s+1})$. 
\end{description}
\par
Set $I_{s-1} = \H^0(X, \mathcal{O}_X(-Z_{s-1}))$.  
Then $I_{s-1}$ is also an Ulrich ideal. 
So we may assume that we can write $I_{s-1} = (u_1,\ldots,u_p,t^{s})$ 
for some minimal set of generators of $\fkm$. 
Since $\fkm (u_1,\ldots,u_p) \subseteq I$ and $\fkm^s \not \subseteq I$, we have that 
$t^s \notin I$. 
Hence by $\ell_A(I_{s-1}/I) = 1$, we can choose an element $a_i \in A$ such that 
$u_i - a_i t^s \in I$ for every $i$. 
By replacing $u_i$ with $u_i - a_it^s$, we may assume that $I'=(u_1,\ldots,u_p,t^{s+1}) \subseteq I$. 
As $\ell_A(I_{s-1}/I')=1$ and $I \ne I_{s-1}$, we can conclude that $I=I'$, as required. 
\end{proof}

\section{Examples} \label{Examples}

Throughout this section, let $k$ be an algebraically closed field 
of characteristic $0$.  
Let $\mathcal{X}_A$  denote the set of nonparameter Ulrich ideals of $A$.

\par 
Let $A$ be a rational double point of type $(A_n)$. 
Then the following example indicates that for any Ulrich module with respect to 
$I$ is a direct summand of $\Syz_A^2(A/I)$. 

\begin{ex} \label{ex-A4}
Let $A=k[[x,y,z]]/(x^2+y^4+z^2) \cong k[[s^4,st,t^4]]$ be a two-dimensional rational 
double point of type $(A_4)$.  
Theorem \ref{Main-RDP} implies that $\calX_A=\{\fkm, I_1\}$, 
where $\fkm =(x,y,z)=(s^4,st,t^4)$ and $I_1=(x,y^2,z)=(s^4,s^2t^2,t^4)$. 
\par
The corresponding anti-nef cycle to $\fkm$ (resp. $I_1$) on the minimal resolution 
is 
\par \vspace{5mm}
\begin{picture}(400,30)(10,0)
 \thicklines
\put(100,10){$Z_{0}=$}
\put(135,18){{\tiny $1$}} 
\put(135,0){{\tiny $E_1$}}
\put(140,12){\circle{8}}
\put(145,12){\line(1,0){20}}
\put(165,18){{\tiny $1$}}
\put(165,0){{\tiny $E_2$}}
\put(170,12){\circle{8}}
\put(175,12){\line(1,0){20}}
\put(195,18){{\tiny $1$}}
\put(195,0){{\tiny $E_3$}}
\put(200,12){\circle{8}}
\put(250,8){$\bigg($}
\put(260,10){resp.}
\put(300,10){$Z_{1}=$}
\put(335,18){{\tiny $1$}} 
\put(335,0){{\tiny $E_1$}}
\put(340,12){\circle{8}}
\put(345,12){\line(1,0){20}}
\put(365,18){{\tiny $2$}}
\put(365,0){{\tiny $E_2$}}
\put(370,12){\circle{8}}
\put(375,12){\line(1,0){20}}
\put(395,18){{\tiny $1$}}
\put(395,0){{\tiny $E_3$}}
\put(400,12){\circle{8}}
\put(410,8){$\bigg).$}
\end{picture}
\par \vspace{3mm} \par \noindent
Moreover, if we put 
\[
M_0=A,\quad M_1 = As+At^3,\quad M_2 = As^2+At^2,\;\; \text{and}\;\; M_3=As^3+At, 
\]
then they are representatives of indecomposable maximal Cohen--Macaulay $A$-modules, and 
thus any maximal Cohen--Macaulay $A$-module can be written as 
\[
 A^{\oplus k} \oplus M_1^{\oplus k_1} \oplus M_2^{\oplus k_2} \oplus M_3^{\oplus k_3}. 
\]
Note that $\Syz_A^2(A/\fkm) \cong M_1 \oplus M_3$ and $\Syz_A^2(A/I_1) \cong M_2^{\oplus 2}$. 
Moreover, Theorem \ref{Main-RDP} says that any Ulrich $A$-module with respect to $\fkm$ (resp. $I_1$)
can be written as 
\[
 M_1^{\oplus k_1} \oplus M_2^{\oplus k_2} \oplus M_3^{\oplus k_3} 
\quad \big(\text{resp.} \; M_2^{\oplus k_2} \big).
\]
\end{ex}

\par
In general, there exists an Ulrich $A$-module with respect to $I$ but not 
a direct summand of $\Syz_A^i(A/I)$. 

\begin{ex} \label{ex-E8}
Let $A=k[[x,y,z]]/(x^3+y^5+z^2)$ be a rational double point of type $(E_8)$. 
Then the set of Ulrich ideals is $\mathcal{X}_A=\{\fkm,I\}$, where $I=(x,y^2,z)$ with $\ell_A(A/I)=2$. 
\par \vspace{5mm}
\begin{picture}(400,35)(0,0)
    \thicklines
\put(90,10){$Z_1=$}
  \put(125,18){{\tiny $4$}}
  \put(125,0){{\tiny $E_1$}}
\put(130,12){\circle*{8}}
\put(135,12){\line(1,0){20}}
  \put(155,18){{\tiny $7$}}
  \put(155,0){{\tiny $E_2$}}
\put(160,12){\circle{8}}
\put(165,12){\line(1,0){20}}
  \put(181,18){{\tiny $10$}}
  \put(183,0){{\tiny $E_3$}}
\put(190,12){\circle{8}}
\put(195,12){\line(1,0){20}}
  \put(215,18){{\tiny $8$}}
  \put(215,0){{\tiny $E_4$}}
\put(220,12){\circle{8}}
\put(225,12){\line(1,0){20}}
  \put(245,18){{\tiny $6$}}
  \put(245,0){{\tiny $E_5$}}
\put(250,12){\circle{8}}
\put(255,12){\line(1,0){20}}
  \put(275,18){{\tiny $4$}}
  \put(275,0){{\tiny $E_6$}}
\put(280,12){\circle{8}}
\put(285,12){\line(1,0){20}}
  \put(305,18){{\tiny $2$}}
   \put(305,0){{\tiny $E_7$}}
\put(310,12){\circle{8}}
\put(190,16){\line(0,1){15}}
  \put(178,35){{\tiny $5$}}
   \put(195,35){{\tiny $E_8$}}
\put(190,34){\circle{8}}
\end{picture}
\par \vspace{3mm}
Let $M_i$ denote the indecomposable maximal Cohen--Macaulay $A$-module 
corresponding to $E_i$ (up to equivalence) via the McKay correspondence 
for every $i=1,\ldots,8$. 
Then we have $\Syz_A^2(A/I) \cong M_1$.  
Indeed, since $\Syz_A^2(A/I)$ is an Ulrich $A$-module with respect to $I$, 
it is isomorphic to $M_1^{\oplus k}$ for some $k \ge 1$. 
Then $k=1$ because $\rank \Syz_A^2(A/I)= \rank M_1=2$. 
\par 
Next we see that $\Syz_A^2(A/\fkm) \cong M_7$. Set $\Omega = \Syz_A^2(A/\fkm)$.   
As $\rank_A \Omega=2$, we have $\Omega \cong M_1$ or $\Omega \cong M_7$. 
It follows from \cite[Corollary 7.7]{GOTWY} that $\Omega \cong M_7$. 
\par 
Similarly, one can easily see that $\Syz_A^i(A/\fkm) \cong M_7$ and $\Syz_A^i(A/I) \cong M_1$ 
for every $i \ge 2$. 
Hence $M_2$ cannot be written as a direct summand of $\Syz_A^i(A/J)$ for any Ulrich ideal $J$.  
\end{ex}

\par
Two-dimensional rational double points $(A_n)$ (Gorenstein quotient singularities) 
admits a sequence of Ulrich ideals of length $m=\lceil \frac{n}{2} \rceil$:
\[
 (x,y^m,z) \subset \cdots \subset (x,y^2,z) \subset (x,y,z).
\]
However the following example shows that each two-dimensional non-Gorenstein cyclic quotient singularity 
has a unique Ulrich ideal (that is, the maximal ideal $\fkm$). 

\begin{ex}[\textbf{Cyclic quotient singularity}] \label{cyclic}
Let $A$ be a two-dimensional cyclic quotient singularity of type $\frac{1}{n}(1,q)$, where 
$q$ and $n$ are integers with $1 < q < n$, $(q,n)=1$. 
Namely, $A$ is the invariant subring of the cyclic group generated by 
\[
 g= \left[
\begin{array}{cc}
\varepsilon_n & 0 \\
0 & \varepsilon_n^q 
\end{array}
\right],
\]
where $\varepsilon_n$ denotes the primitive $n$th root of $1 \in k$. 
\par  
Now suppose that $A$ is not Gorenstein, that is, $q+1$ is not divided by $n$. 
Then there exists an exceptional curve $E_i$ so that $b:=-E_i^2 \ge 3$. 
In particular, $KE_i = b-2 \ge 1$. 
Let $\Gamma$ be the dual graph of the minimal resolution of singularities $X \to \Spec A$.
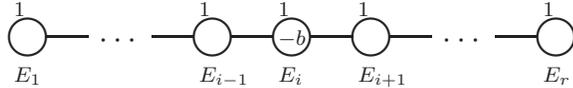
\par \vspace{4mm}
\begin{picture}(400,40)(0,0)
    \thicklines
\put(55,35){{\tiny $1$}}
\put(55,10){{\tiny $E_1$}}
\put(60,27){\circle{14}}
\put(67,27){\line(1,0){16}}
\put(87,23){$\cdots$}
\put(107,27){\line(1,0){16}}
\put(130,27){\circle{14}}
\put(125,35){{\tiny $1$}}
\put(125,10){{\tiny $E_{i-1}$}}
\put(137,27){\line(1,0){16}}
\put(160,27){\circle{14}}
  \put(155,24){{\tiny $-b$}}
\put(155,10){{\tiny $E_i$}}
\put(155,35){{\tiny $1$}}
\put(167,27){\line(1,0){16}}
\put(190,27){\circle{14}}
\put(185,35){{\tiny $1$}}
\put(185,10){{\tiny $E_{i+1}$}}
\put(197,27){\line(1,0){16}}
\put(217,23){$\cdots$}
\put(237,27){\line(1,0){16}}
\put(260,27){\circle{14}}
\put(255,35){{\tiny $1$}}
\put(255,10){{\tiny $E_r$}}
\end{picture}
\par \vspace{4mm}
It is well-known that $\fkm$ is an Ulrich ideal. 
Now suppose that there exists an Ulrich ideal other than $\fkm$. 
Then we can take $Y_1$ satisfying the conditions (3)(a)(b) in Theorem \ref{All-Uideal}.
In particular, $Z_0Y_1=0$ and $K(Z_0-Y_1)=0$ and $0 < Y_1 \le Z_0$. 
Set $Y_1 = \sum_{j \in J} E_j$ for some non-empty subset $J$ of $\{1,\ldots,r\}$. 
\par
If $i \in J$, then $Z_0E_i=0$ because $Z_0Y_1=0$. 
On the other hand, $Z_0E_i \le E_i^2+2 =2-b \le -1$. This is a contradiction. 
Hence $i \notin J$. Then $E_i \subset \Supp(Z_0-Y_1)$. 
This implies that $KE_i=0$ because $K(Z_0-Y_1)=0$, 
which contradicts the choice of $E_i$.
Hence the maximal ideal is the only Ulrich ideal of $A$. 
\end{ex}

\begin{rem}
Let $A$ be a cyclic quotient singularity as above. 
Then one can obtain many examples of special cycles in general
by a similar argument to the proof of Theorem $\ref{Main-RDP}$. 
\end{rem}

\begin{ex}[\textbf{Rational triple points}] \label{abc}
Let $a \ge b \ge c \ge 2$. 
If we set 
$A=k[[T,sT^a,s^{-1}T^b,(s+1)^{-1}T^c]]$, 
then it is a two-dimensional 
rational singularity with $\e_{\fkm}^0(A)=3$ and 
\begin{eqnarray*}
A &\cong & 
k[[t,x,y,z]]/(xy-t^{a+b},xz-t^{a+c}+zt^a,yz-yt^c+zt^b)\\[2mm]
& \cong & k[[t,x,y,z]]\bigg/I_2 \left(\begin{array}{ccc} x & t^b & t^c-z \\ t^a & y & z \end{array} \right).
\end{eqnarray*}
Then $I_k=(t^k,x,y,z)$ is an Ulrich ideal of colength $k$ 
for every $k$ with 
$1 \le k \le c$.
In fact, if we put $Q_k=(t^k,x+y+z)$, then $I_k=Q_k+(x,y)$ and $I_k^2=Q_kI_k$. 
Furthermore, we have $e_{I_k}^0(A)=\ell_A(A/Q_k) = 3k = (\mu(I)-1)\cdot \ell_A(A/I)$. 
Hence $I_k$ is an Ulrich ideal. 
\par \vspace{2mm}
Let $a=b=c=3$. 
Now consider the corresponding cycles. 
 
\par 
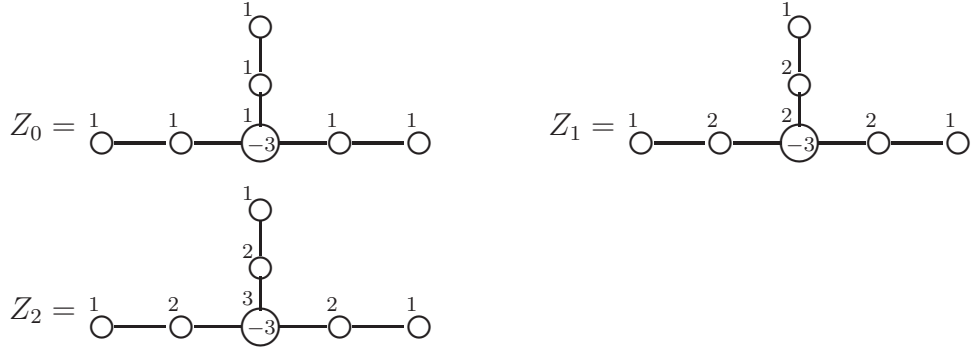
\vspace{3mm}
\begin{picture}(200,60)(-40,0)
    \thicklines
 \put(-5,10){$Z_0=$}
\put(30,7){\circle{8}}
   \put(25,13){{\tiny $1$}}
\put(35,7){\line(1,0){19}}
  \put(55,13){{\tiny $1$}}
\put(60,7){\circle{8}}
\put(65,7){\line(1,0){18}}
  \put(85,4){{\tiny $-3$}}
  \put(83,15){{\tiny $1$}}
\put(90,7){\circle{14}}
\put(97,7){\line(1,0){18}}
  \put(115,13){{\tiny $1$}}
\put(120,7){\circle{8}}
\put(125,7){\line(1,0){19}}
  \put(145,13){{\tiny $1$}}
\put(150,7){\circle{8}}
\put(90,13){\line(0,1){13}}
  \put(83,33){{\tiny $1$}}
\put(90,29){\circle{8}}
\put(90,34){\line(0,1){13}}
  \put(83,55){{\tiny $1$}}
\put(90,51){\circle{8}}
\end{picture} 
\begin{picture}(200,60)(-40,0)
    \thicklines
 \put(-5,10){$Z_1=$}
\put(30,7){\circle{8}}
   \put(25,13){{\tiny $1$}}
\put(35,7){\line(1,0){19}}
  \put(55,13){{\tiny $2$}}
\put(60,7){\circle{8}}
\put(65,7){\line(1,0){18}}
  \put(85,4){{\tiny $-3$}}
  \put(83,15){{\tiny $2$}}
\put(90,7){\circle{14}}
\put(97,7){\line(1,0){18}}
  \put(115,13){{\tiny $2$}}
\put(120,7){\circle{8}}
\put(125,7){\line(1,0){19}}
  \put(145,13){{\tiny $1$}}
\put(150,7){\circle{8}}
\put(90,13){\line(0,1){13}}
  \put(83,33){{\tiny $2$}}
\put(90,29){\circle{8}}
\put(90,34){\line(0,1){13}}
  \put(83,55){{\tiny $1$}}
\put(90,51){\circle{8}}
\end{picture} 
\par \vspace{3mm}
\begin{picture}(200,60)(-40,0)
    \thicklines
 \put(-5,10){$Z_2=$}
\put(30,7){\circle{8}}
   \put(25,13){{\tiny $1$}}
\put(35,7){\line(1,0){19}}
  \put(55,13){{\tiny $2$}}
\put(60,7){\circle{8}}
\put(65,7){\line(1,0){18}}
  \put(85,4){{\tiny $-3$}}
  \put(83,15){{\tiny $3$}}
\put(90,7){\circle{14}}
\put(97,7){\line(1,0){18}}
  \put(115,13){{\tiny $2$}}
\put(120,7){\circle{8}}
\put(125,7){\line(1,0){19}}
  \put(145,13){{\tiny $1$}}
\put(150,7){\circle{8}}
\put(90,13){\line(0,1){13}}
  \put(83,33){{\tiny $2$}}
\put(90,29){\circle{8}}
\put(90,34){\line(0,1){13}}
  \put(83,55){{\tiny $1$}}
\put(90,51){\circle{8}}
\end{picture} 
\par \vspace{5mm}
Note that all special cycles are Ulrich cycles. 
%
\end{ex}

\begin{ex}
Let $A=k[[s^7,s^4t,st^2,t^7]]$. 
Then $A$ is a two-dimensional cycle quotient singularity, which is an invariant subring 
of a cyclic group generated by 
\[
g= \left(
\begin{array}{cc}
\varepsilon_7 & \\
 & \varepsilon_7^3
\end{array}
\right).
\]
Since $7/3 = 7 - \dfrac{1}{2-1/2}$, 
the dual graph can be written as the following form:
\par 
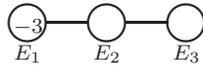
\vspace{2mm}
\begin{picture}(400,30)(10,0)
 \thicklines
\put(135,8){{\tiny $-3$}} 
\put(135,-2){{\tiny $E_1$}}
\put(140,12){\circle{14}}
\put(147,12){\line(1,0){16}}
\put(165,-2){{\tiny $E_2$}}
\put(170,12){\circle{14}}
\put(177,12){\line(1,0){16}}
\put(195,-2){{\tiny $E_3$}}
\put(200,12){\circle{14}}
\end{picture}
\par \vspace{3mm} \par \noindent
If we put $N_a=\langle s^it^j \,|\,i+3j \equiv a \pmod{3} \rangle$ for $a=0,1,\ldots,6$, then 
$\{N_a\}_{a=0}^{6}$ forms a representative of isomorphism classes of indecomposable 
maximal Cohen--Macaulay $A$-modules. 
Then $M_1 = N_3=As+At^5$, $M_2 = N_2 = As^2+At^3$ and $M_3 = N_1 = As^3+At$ 
are indecomposable special Cohen--Macaulay $A$-modules. 
On the other hand, $N_4=As^4+Ast+At^6+$, $N_5=As^5+As^2t+At^4$, $N_6=As^6+As^3t+At^2$ 
are indecomposable Ulrich $A$-modules with respect to the maximal ideal $\fkm$.  
\par  
All special cycles are $Z_0=E_1+E_2+E_3$ and $Z_1 = E_1+2E_2+E_3$. 
(Note that $Z_1$ is not an Ulrich cycle; see Example \ref{cyclic}).
Any special Cohen--Macaulay module with respect to $I_{Z_1}$ 
is of the form $M_2^{\oplus k}$. 
\par 
However, we do not have the complete list of Ulrich modules with respect to some ideal. 
\end{ex}


\end{document}